\documentclass[12pt,draft]{amsart}
\usepackage{graphicx}
\usepackage{amsmath,amsbsy}
\usepackage{amssymb, enumerate,color}
\usepackage{comment}
\usepackage{mathabx}
\usepackage{mathrsfs}      
\usepackage{helvet}         
\usepackage{courier}        
\usepackage{type1cm}        

\includecomment{versiona}

\definecolor{ForestGreen} {cmyk}{0.91,0,0.88,0.12}
\newcommand{\rosso}[1]{{\color{red} #1}} 
\newcommand{\verde}[1]{{\color{ForestGreen} #1}} 

\long\def\COMMENT#1{\vskip14pt\par\vbox{\vskip2pt\hrule\vskip14pt\rosso{#1}\vskip14pt\hrule\vskip2pt}\vskip14pt} 
\long\def\NOTE#1{{}}

\usepackage{geometry,calc}
\newtheorem{thm}{Theorem}[section]
\newtheorem{cor}[thm]{Corollary}
\newtheorem{lem}[thm]{Lemma}
\newtheorem{prop}[thm]{Proposition}

\newtheorem{pb}[thm]{Problem}
\newtheorem{assmp}[thm]{Assumption}
\theoremstyle{remark}
\newtheorem{rem}[thm]{Remark}
\numberwithin{equation}{section}

\usepackage{esint}

\renewcommand{\O}{\Omega}

\newcommand{\Real}{\mathbb R}

\newcommand{\grad}{\nabla}

\newcommand{\Zperp}{\widehat\Lambda}

\newcommand{\Cs}{c_s}

\newcommand{\Astrong}{\mathscr{A}}
\newcommand{\Vfull}{\mathscr{V}}
\newcommand{\ffull}{\mathsf{f}}
\newcommand{\ufull}{\mathsf{u}}
\newcommand{\vfull}{\mathsf{v}}
\newcommand{\wfull}{\mathsf{w}}

\newcommand{\TG}{\mathcal{K}_\Gamma}

\newcommand{\Thl}{\mathcal{T}_\delta^k}
\newcommand{\Gl}{\Gamma_k}
\newcommand{\El}{\mathcal{K}^k}

\newcommand{\hatV}{\widehat V}
\newcommand{\Rk}{\opR _k}
\newcommand{\opA}{A}
\newcommand{\opB}{B}
\newcommand{\opR}{\mathscr{R}}
\newcommand{\opG}{G}
\newcommand{\opS}{S}

\newcommand{\Poly}[1]{\mathbb{P}_{#1}}

\newcommand{\smo}{r} 
\newcommand{\Vsp}{W^{\smo}}
\newcommand{\Ls}{\Xi^{\smo}}
\newcommand{\Lt}{\Xi^t}

\newcommand{\carrierAk}{C_k}

\newcommand{\xid}{\xi_\delta}

\newcommand{\Lh}{\Lambda_\delta}
\newcommand{\tLh}{\widetilde \Lambda_\delta}
\newcommand{\hLh}{\widehat \Lambda_\delta}

\newcommand{\Range}[1]{\mathfrak{R}(#1)}

\newcommand{\Bd}{B_\delta}
\newcommand{\Gd}{G}

\newcommand{\Bh}{\Bd}
\newcommand{\opNorm}[3]{\| #3 \|_{#1\to#2}}

\newcommand{\Vh}{V_\delta}
\newcommand{\Wh}{V_h}

\newcommand{\hSh}{\widehat S_h}

\newcommand{\Vscal}{{d}}

\newcommand{\tAh}{\widetilde A_\delta}

\newcommand{\dsc}[2]{[#1,#2]_\delta}

\newcommand{\tLhs}{\tilde \Lambda_{\hS}}

\newcommand{\lh}{\lambda_\delta}
\newcommand{\mh}{\mu_\delta}

\newcommand{\hO}{h}
\newcommand{\hS}{\delta}
\newcommand{\thS}{\tilde \hS}

\newcommand{\tpiS}{\tilde\pi_\delta}

\newcommand{\xh}{\xi_\delta}
\newcommand{\sh}{s_h}
\newcommand{\eh}{e_h}

\newcommand{\hM}{\widehat \M}

\newcommand{\Embed}{\mathcal K}
\newcommand{\tphi}{{\widetilde\varphi}}
\newcommand{\hlambda}{\widehat \lambda}
\newcommand{\hmu}{\widehat\mu}

\newcommand{\tpsi}{\widetilde \psi}
\newcommand{\hm}{\m^*}

\newcommand{\Inject}{\mathcal{I}}
\newcommand{\Restrict}{\mathcal{R}}
\newcommand{\Id}[1]{\mathbf{1}_{{#1}}}

\newcommand{\Ah}{A_\delta}

\newcommand{\Lm}{\Lambda^\ell}

\newcommand{\kplus}{{k^+_\ell}}
\newcommand{\kminus}{{k^-_\ell}}
\newcommand{\boundary}[1]{\partial(#1)}
\newcommand{\neighbors}[1]{\mathcal{N}(#1)}

\newcommand{\cVsplit}{c^\flat}
\newcommand{\CVsplit}{C^\flat}

\newcommand{\tsigma}{\sigma^*}
\newcommand{\Qsigma}{Q_\sigma^*}
\newcommand{\QsigmaT}{(\Qsigma)^T}

\newcommand{\Qd}{Q_d}
\newcommand{\ta}{\widetilde a}

\newcommand{\normta}[1]{\| #1 \|_{\ta}}

\newcommand{\D}{\mathcal{D}}

\newcommand{\G}{\Gamma}

\begin{document}

\title[Black Box Coupling]{An abstract framework for heterogeneous coupling: stability, approximation and preconditioning}%
\author{Silvia Bertoluzza}%
\address{IMATI-CNR}%
\email{silvia.bertoluzza@imati.cnr.it}%

\author{Erik Burman}%
\address{Department of Mathematics, University College London, United Kingdom}%
\email{e.burman@ucl.ac.uk}%

\begin{abstract}
We consider heterogeneous coupling problems on an abstract level, establishing fundamental principles of domain decomposition agnostic to the solvers of the local subproblems.
Introducing a coupling framework reminiscent of FETI methods, but here on abstract form, we establish conditions for stability and minimal requirements for well-posedness on the continuous level, as well as conditions on local solvers for the approximation of subproblems. We then discuss stability of the resulting Lagrange multiplier methods and show stability under a mesh condition between the local discretizations and the mortar space. If this condition is not satisfied we show how a stabilization, acting only on the multiplier can be used to achieve stability. The design of preconditioners of the Schur complement system is discussed in the unstabilized case. Finally we discuss some applications that enter the framework.

\end{abstract}
\maketitle

\section{Introduction}

The efficient approximation of linear problems set in infinite dimensional spaces is becoming increasingly important. A key example is the numerical approximation of solutions to partial differential equations (PDE) or integral equations, that are omni present in science and technology.  
Discretization results in linear systems that become increasingly large as the scale and the complexity of the problem grow. In view of the increasing availability of high performance computing environments, it is then advantageous to resort to solution strategies leveraging, possibly iteratively, local solvers for sets of (suitably coupled) subproblems.  

The common  approach in this framework is to discretize the global continuous problem into a  possibly huge algebraic system, and then split the latter as a system of coupled algebraic subsystems. In this case the coupling condition often consists in the identification of unknowns in the subsystems corresponding to the same global unknown. An alternative approach, which is lately gaining increasing interest, is, instead, to decompose the original problem already at the continuous/infinite dimensional level, thus obtaining a system of coupled infinite dimensional subproblems, to be successively discretized. This is the approach that underlies a number of non conforming domain decomposition methods, such as the mortar method \cite{bernardi2001spectral,belgacem1999mortar,Wohl00}, the three fields domain decomposition method \cite{brezzi2001error,bertoluzza2003analysis}, or the domain decomposition method  based on local Dirichlet--Neumann maps proposed in \cite{steinbach2005natural}. 

A key feature of this approach, that we will exploit herein, is that the subproblems can use different models and be discretized independently of each other using bespoke solvers that communicate only through the coupling conditions. This offers the possibility of resorting to independent implementations for the solvers of the different subproblems, an appealing feature, particularly  in the framework of multiphysics and multiscale simulations, where highly specialized and optimized  subproblem solvers might be available, and where the whole coupled problem might be excessively complex. 
Indeed,
 in this context, we see a growing interest around the different issues related to ``code coupling" \cite{keyes2013multiphysics,groen2013survey,longshaw2017code}, and the development of different computational environments aimed at enabling the interaction of existing solvers, individually simulating different portions of a larger complex physical problems (see, for instance, \cite{tang2015multiscale, bungartz2016precice,veen2020easing}). 

The aim of this paper is to provide a theoretical foundation for such approaches where heterogeneous large scale problems are solved using different highly optimized codes to handle the different subproblems. Therefore we study a strategy, agnostic to the different solvers and subproblems, 
in an abstract framework. The problem, the subproblems and coupling are not a priori associated to a physical space and model, and therefore a wide range of applications enter the framework. 
For the local solvers we only assume that they produce a sufficiently accurate approximation of {the coupling quantities}.  
Following an approach  inspired by the widely used Finite Element Tearing and Interconnecting (FETI) method \cite{farhat1991method,FCR94,FMR94,LM97},  by formally solving the local problems, up to a possible non-trivial, but finite dimensional, kernel of the local operator, the system is  reduced to a problem  where only the unknowns enforcing the coupling between subproblems, that is the Lagrangian multipliers, are to be computed. While the FETI method was developed in the framework of conforming domain decomposition for PDEs, the analysis herein applies to a large class of other problems, including multiscale models \cite{BD98}, FEM/BEM coupling \cite{BF19}, or PDEs set on lower dimensional manifolds connected in a network. It also provides a roadmap to techniques for solution of the resulting algebraic system. 

Our abstract starting point allows us to remain agnostic as to the local approximation strategy to the furthest possible extent: discretization of the local subproblem consists, in our abstract famework, in selecting a non-specified ``local solver". This important feature of our approach   
allows the coupling of highly optimized codes with solvers that use different methods, for instance finite difference methods, finite element methods, spectral methods or boundary element methods.
 By carrying out our analysis in an abstract framework, independent of the discretization methodology, we establish the weakest conditions that such solvers must satisfy in order to guarantee the well posedness of the resulting global method. To establish such minimal conditions we analyze the properties of the reduced coupling system, for which well-posedness is proven at  the continuous level. This continuous framework is then exploited when local discretizations are introduced, which we treat as ``black boxes'', and conditions  are proposed that are sufficient (but not necessary) to establish uniform stability of the discretization. To cater for the situation where such conditions, which implicitely  enforce a compatibility condition between local solvers and multiplier discretization,
 can not be satisfied or lead to inconvenient discretizations, we introduce stabilization methods in the spirit of \cite{BF01}, generalizing the ideas of \cite{BH10, Bu14}.
We establish optimal error estimates for the discretization and we discuss the design of preconditioners for the unstabilized reduced system. A particularly salient feature of the black box framework is that local solvers based on non-coercive formulations can be preconditioned using alternative coercive formulations. This observation points to a new class of preconditioners for saddle point problems based on local stabilised formulations. The two canonical examples of domain decomposition, Neumann-Neumann and Dirichlet-Dirichlet coupling are used as applications to illustrate the theory.

\section{Problem setting}
We consider an abstract problem of the form
\begin{equation}\label{abstractStrong}
\Astrong \ufull = \ffull, \quad \text{ in } \Vfull',
\end{equation}
where $\Vfull$ is a separable Hilbert space, $\Astrong : \Vfull \to \Vfull'$ denotes a boundedly invertible  linear operator, and where $f \in \Vfull'$ is the given right hand side. 
We assume that such a problem can be split as a system of ``local'' subproblems on smaller Hilbert spaces $V_k$, $k=1,\ldots,N$ subject to a coupling constraint. 
More precisely, we assume that we are given injective ``restriction'' operators  $\Rk: \Vfull \to V_k$, $k = 1,\cdots,N$, and we define the composite operator $\opR$ which maps $\Vfull$  into $\prod_{k=1}^N V_k$,  as
	\[
	\opR  = \left(
	\begin{array}{c}
		\opR _1\\
		\vdots\\
		\opR _N
	\end{array}
	\right). \]
We then assume that there exist operators $\opB _k : V_k \to \Lambda'$, $\Lambda$ being  a second separable Hilbert space, such that the operator $\opR$
is an isomorphism of $\Vfull$ onto the space $\hatV \subseteq V =  \prod_{k=1}^N V_k$, defined as
\[
\hatV =  \{ v = (v_k)_k \in V : \sum_k \opB _k v_k = 0 \ \text{ in } \Lambda'\}.
\]
We let $b_k: V_k \times \Lambda \to \Real$ denote the   bilinear form corresponding to $\opB_k$:
\[ b_k (u,\mu) =
\langle \opB _k u , \mu \rangle, \qquad \forall \mu \in \Lambda.
\]
{Above and in the following we use the notation $\langle \cdot,\cdot \rangle$ to denote different duality pairings (the spaces paired follow from the context in each case).}

\

Assume now that continuous bilinear forms $a_k : V_k \times V_k \to \Real$ and linear operators $f_k \in V_k'$ are given,  such that for all $\vfull, \wfull \in \Vfull$ it holds that
\[
\langle \Astrong \vfull, \wfull \rangle = \sum_{k=1}^N a_k (\Rk \vfull, \Rk \wfull ), \qquad \langle \ffull, \wfull \rangle = \sum_{k=1}^N \langle f_k,\Rk \wfull \rangle,
\]
and let $\opA _k: V_k \to V_k'$ denote the linear operator corresponding to $a_k$:
\[
\langle \opA _k u, v \rangle = a_k(u,v), \qquad \forall u, v \in V_k.
\]
We assume that $\ker \opA_k$ is finite dimensional and coincides with $\ker \opA_k^T$, and that $a_k$ is coercive on the subspace \[\carrierAk = (\ker \opA_k)^\perp \subset V_k.\]

Letting
\[
\opA  = \left[
\begin{array}{ccc}
  \opA _1 & 0 & 0\\
  0& \ddots & 0 \\
  0 & 0 & \opA _N
\end{array}
\right], \qquad  \opB  = [\opB _1, \cdots ,\opB _N], \qquad f = \left[
\begin{array}{c}
f_1\\
\vdots\\
f_N
\end{array}
\right],
\]
we can  finally rewrite Problem \ref{abstractStrong} in the following saddle point form
\begin{pb}\label{saddle1} Find $u \in V$ and $\lambda \in \Lambda$
such that
\begin{eqnarray}
\label{eq:saddle1}  \opA  u - \opB ^T \lambda  &=& f, \\
\label{eq:saddle2}  \opB  u &=& 0.
\end{eqnarray}
\end{pb}

\

Letting $\| \cdot \|_V$ denote the product norm in $V$, we assume that an inf-sup condition for $\opB $ holds, of the form\footnote{Here and in the following we use the notation $ x \lesssim y$ to indicate that the quantity $x$ is lower that the quantity $y$ times a positive constant independent on any relevant  parameter, particularly the ones related to the forthcoming discretizations.}
\begin{equation}\label{infsupbk}
\inf_{\lambda \in \Lambda} \sup_{v = (v^k) \in V} \frac{b(v,\lambda)}{\| v \|_{V} \| \lambda \|_\Lambda} \gtrsim 1, \qquad \text{ where } b(v,\lambda) = \sum_k b_k(v^k,\lambda).
\end{equation}
We do not assume that $\opA _k$ is invertible for all $k$, however the well posedness of equation \eqref{abstractStrong}, together with the above inf-sup condition, implies the well posedness of Problem \ref{saddle1}.

\begin{rem}
Observe that one can also consider a slightly more general form of Problem \ref{saddle1}, where $g \in \Lambda'$ and $C:\Lambda \to \Lambda'$ is a positive semi definite linear operator. 
\begin{pb}\label{pb:Robin1} Find $u \in V$ and $\lambda \in \Lambda$
such that
\begin{eqnarray}
\label{eq:Robin1}  \opA  u - \opB ^T \lambda  &=& f, \\
\label{eq:Robin2}  C \lambda + \opB  u &=& g.
\end{eqnarray}
\end{pb}
In the PDE context this allows to include Robin type interface coupling conditions between subdomains. The below discussion of well posedness and discretization extends directly to this case.
\end{rem}

%
%
%

\

In order to define the  numerical method aimed at solving our problem, we follow, already at the continuous level, the approach underlying the FETI  domain decomposition method \cite{farhat1991method}. We start by eliminating the unknown $u$. 
To this aim for all $k$ we introduce a bounded linear pseudo inverse $\opA ^+_k$ of $\opA_k$, $\opA_k^+ : V_k' \to \carrierAk \subset V_k$,
defined as follows
\begin{equation}\label{eq:pseudo_inverse_def}
a_k(\opA ^+_k g , w) = \langle g , w \rangle, \qquad \text{ for all }w \in \carrierAk.
\end{equation}
Letting $\Range A$ denote the range of $A$, such a pseudo inverse is well defined and verifies
\[\opA _k\opA _k^+ g = g, \quad \text{ for all }g \in \Range {A_k},\qquad\text{and}
\qquad\| \opA _k^+ A v \|_{V_k} \lesssim \|  A_k v \|_{V'_k}, \qquad \forall v \in V_k.\]


%
%
%

 We then assemble the linear operator $A^+ : V' \to V$ as 
\[
\opA ^+ = \left[
\begin{array}{ccc}
  \opA ^+_1 & 0 & 0\\
  0& \ddots & 0 \\
  0 & 0 & \opA ^+_N
\end{array}
\right].
\]
Letting the bilinear operator $a: V \times V \to \mathbb{R}$  be defined as 
\[\qquad a(u,v) = \langle Au, v \rangle = \sum_k a_k(u^k,v^k),\]
we can easily see that we have 
\[
a(A^+ Av,A^+ A v) \gtrsim \| A^+ A v \|_{V}^2, \qquad \text{ for all } v \in V.
\]

Equation \eqref{eq:saddle1} trivially implies that $f + \opB ^T \lambda = \opA u 
$. Then,
 letting \[Z = \ker \opA  \subset V,\] we immediately see that for some unknown $z \in Z$, the solution $(u,\lambda)$ of Problem \ref{saddle1} satisfies
\begin{equation}\label{eq:u_def}
u = \opA ^+(f + \opB ^T \lambda) + z.
\end{equation}

Plugging this expression in \eqref{eq:saddle2} we obtain an identity in $\Lambda'$ of the form
\[
\opB \opA ^+\opB ^T \lambda  + \opB  z = -\opB \opA ^+ f.
\]
 If we  test \eqref{eq:saddle1} with any  $z \in Z$, observing that, since, by the assumption that $\ker \opA = \ker \opA ^T$, we have 
\[\langle \opA  u , z \rangle = \langle u , \opA ^T z \rangle = 0,\]
we immediately obtain that 
\begin{equation}\label{kerequation}
\langle \opB ^T \lambda + f , z \rangle = 0, \ \forall z \in Z.
\end{equation}
 Letting $ G : Z \to \Lambda'$ denote the restriction to $Z$ of $\opB $, we can rewrite \eqref{kerequation} as an identity in $Z'$:
\[\opG ^T \lambda + f = 0,\]
where, by abuse of notation, we let $f \in Z'$ denote the restriction of $f$ to $Z$.

\
\

In other words we transformed the original saddle point problem to the following reduced saddle point problem, whose unknowns are the multiplier $\lambda$ and the $\ker \opA$ component $z$ of the solution. 
\begin{pb}\label{schur} 
 Given $f \in V'$ find $\lambda \in \Lambda$ and $z \in Z$
such that
\begin{eqnarray*}
  \opS  \lambda + \opG  z  &=& g, \\
  \opG ^T \lambda &=& -f, 
\end{eqnarray*}
where $\opS  = \opB  \opA ^+ \opB ^T$ and $g = -\opB  \opA ^+ f$.
\end{pb}

It will be convenient, in the following, to write Problem \ref{schur} in weak form as: find $\lambda \in \Lambda$ and $z \in Z$
such that
\begin{eqnarray*}
  s(\lambda ,\mu ) + b (z,\mu )  &=& \langle g,\mu  \rangle, \qquad \forall \mu \in \Lambda, \\
  b(w, \lambda ) &=& -\langle f, w \rangle \qquad \forall w \in Z.
\end{eqnarray*}
with $s:\Lambda \times \Lambda \to \mathbb{R}$ defined by
\[s(\lambda,\mu) = \langle \opS  \lambda , \mu \rangle.\]

Note that the presence of a nontrivial kernel $Z$ indicates that a global coupling between the subproblems remains.
The cardinality of the finite dimensional set $Z$ is in this sense a measure of how efficient the decoupling in local subproblems is. We will  assume below that the dimension of $Z$ is relatively small, implying that the decomposition at hand is efficient.

\


This new saddle point problem is the starting point of our abstract coupling algorithm. We start by verifying the well posedness of such a problem. We have the following Theorem.
\begin{thm}\label{thm:wellposedness_cont}
 The following {\em inf-sup} condition holds
\begin{equation}\label{continfsup}
\inf_{z \in Z} \sup_{\lambda \in \Lambda} \frac {\langle \opB z,\lambda\rangle }{\| z \|_{V} \|\lambda \|_{\Lambda}} = 
\inf_{z \in Z} \sup_{\lambda \in \Lambda} \frac {b(z,\lambda)}{\| z \|_{V} \|\lambda \|_{\Lambda}} \gtrsim 1.
\end{equation}
Moreover the operator $s$ satisfies 
\begin{equation}\label{contelker}
 s(\lambda,\lambda) \gtrsim \| \lambda \|^2_{\Lambda} \quad \text{ for all }\lambda \in\Zperp= \ker \opG^T =  \{ \lambda \in \Lambda:\  b(z, \lambda) = 0\ \forall z \in Z\}
\end{equation}
and 
\begin{equation}\label{Scont}
 s(\lambda,\mu) \lesssim \| \lambda \|_{\Lambda} \| \mu \|_{\Lambda}.
 \end{equation}
\end{thm}

\begin{proof}
We prove (\ref{continfsup}) by contradiction: more precisely we assume that for each $\epsilon > 0$ there exist $z_\epsilon  \in Z$
with $\| z_\epsilon \|_{V} = 1$ such that for all $\lambda$ with $\| \lambda \|_{\Lambda}=1$ it holds that $b(z_\epsilon,\lambda) < \epsilon$.
 We can then extract a weakly convergent sequence $z_n$. Let $z$ denote the weak limit of the sequence: $z_n \rightharpoonup z$.
By definition of weak convergence we immediately see that $z \in \ker \opB $. Now we recall that $\ker \opB  \cap Z = \ker \opB  \cap \ker \opA  = \{ 0 \}$. As $Z$ is finite dimensional, this is in contradiction with $\| z_n \|_V = 1$.

\

We next prove the coercivity of $s$ on $\widehat \Lambda$.
We start by observing that
$\opB ^T \Zperp\subseteq \Range A$.  Indeed, as, by assumption, $a$ is coercive in $\Range {A^+}$, there exist $w \in \Range {A^+}$ such that for all $v \in \Range{A^+}$ it holds that
\[
a(w,v) = \langle B^T \lambda , v \rangle, \qquad \| w \|_{V} \lesssim \| B^T \lambda \|_{V'} \lesssim \| \lambda \|_{\Lambda}.
\]
We claim that $A w = B^T \lambda$ in $V'$. Indeed, for all $v \in V$ we have $v = A^+ A v + z$ with $z\in Z$, so that we can write
\begin{multline}
	\langle \opA w,v  \rangle  = a(w,A^+ A v + z) = a(w, A^+ A v) = \langle B^T \lambda , A^+ A v \rangle + \langle G^T \lambda , z \rangle = \langle B^T \lambda, v \rangle.
\end{multline}

Now, given $\lambda \in \Zperp$ we have
\[
s(\lambda,\lambda) = \langle \opB  \opA ^+ \opB ^T \lambda, \lambda \rangle = \langle \opA ^+ \opB ^T \lambda , \opB ^T \lambda \rangle.
\]
On the other hand, as $B^T \lambda \in \Range A$, we have
\[
\| \opB ^T \lambda \|_{V'} = \| \opA  \opA ^+ \opB ^T \lambda \|_{V'} \lesssim \| \opA ^+ \opB ^T \lambda \|_{V}.
\]
Then we can write
\begin{gather*}
	s( \lambda, \lambda ) =  a( \opA ^+ \opB ^T \lambda ,
	\opA ^+ \opB ^T \lambda) \gtrsim \| \opA ^+ \opB ^T \lambda \|_{V}^2
	\gtrsim \| \opB ^T \lambda \|_{V'}^2.
\end{gather*}

We have, thanks to the inf-sup condition \eqref{infsupbk}
\[
\| \opB ^T \lambda \|_{V'} = \sup_{v \in V} \frac{ b (v ,
	\lambda) }{\| v \|_{V}} \gtrsim \| \lambda \|_{\Lambda},
\]
which allows us to conclude that
\[
s(\lambda , \lambda) \gtrsim \|
\lambda \|^2_{\Lambda}.
\]

The continuity \eqref{Scont} follows, since
\[
s(\lambda , \mu) = \langle   \opA ^+ \opB ^T \lambda, \opB^T \mu \rangle \lesssim \| {\opA ^+} \opB ^T \lambda\|_V 
\|  \opB ^T \mu\|_{{V'}}\lesssim  \|
\lambda \|_{\Lambda}\|
\mu \|_{\Lambda}.
\]
\end{proof}

Classical results on saddle point problems \cite{boffi2013mixed}, together with the well posedness of Problem \ref{abstractStrong},  imply the following corollary.
\begin{cor}
Problem \ref{schur} admits a unique solution $(\lambda,z)$ such that, letting $\ufull$ denote the solution of (\ref{abstractStrong}), we have \[\opR \ufull = \opA ^+ (\opB ^T \lambda + f)+ z.\]
\end{cor}

\section{Discretization}
We discretize Problem \ref{schur} by a 
Galerkin discretization, involving both approximating the multiplier space $\Lambda$ with a finite dimensional space $\Lambda_\delta$, and numerically evaluating the pseudo inverses $A_k^+$.  More precisely, for all $k$ we consider approximate linear solvers $\opA ^+_{h,k}: V' \to \carrierAk$, that we will treat as a black-boxes. We then let
\[
\opA ^+_h = \left[
\begin{array}{ccc}
\opA ^+_{h,1} & 0 & 0\\
0& \ddots & 0 \\
0 & 0 & \opA ^+_{h,N}
\end{array}
\right],
\]
and we define an approximate bilinear form $s_h:\Lambda \times \Lambda \to \mathbb{R}$ and an approximate right hand side $g_h$ as
\[
s_h(\lambda,\mu) = \langle \opB  \opA ^+_h \opB ^T \lambda, \mu \rangle, \qquad g_h = -BA^+_h f.
\]

Let $\Lambda_\delta  \subseteq \Lambda$ be a finite dimensional subspace depending on a
 parameter $\delta > 0$, playing the role of a meshsize. Since the space $Z = \ker A$ is, by assumption, finite dimensional,
we do not need to discretize it. We then consider the following discrete problem.
\begin{pb}\label{feti_discrete} Find $\lh  \in \Lambda_\delta$ and $z^* \in Z$
such that
\begin{eqnarray*}
  s_h(\lh ,\mh ) + b (z^*,\mh )  &=& \langle g_h,\mh  \rangle, \qquad \forall \mh  \in \Lambda_\delta, \\
  b(w, \lh ) &=& -\langle f, w \rangle, \qquad \forall w \in Z.
\end{eqnarray*}
\end{pb}
\


As we want to treat the local solvers as black boxes, we want the assumptions on the approximate pseudo inverses to be as weak as possible.  
We want to understand under which minimal conditions Problem \ref{feti_discrete}  is an approximation to Problem \ref{schur}. To this aim it is only natural to require that the  black box approximate linear solvers $\opA ^+_{h,k}$ are  approximations of the continuous pseudo inverses $\opA^+_k$, at least when the data have some extra smoothness. 

\

To quantify, in our abstract setting, the concept of ``extra smoothness'', we  introduce infinite dimensional subspaces of $V_k$ and $\Lambda$, playing the role of higher regularity spaces. More precisely we  let $\Vsp_k \subseteq V_k$ and $\Ls \subseteq \Lambda$, $\smo \geq 0$, denote a scale of subspaces of $V_k$ and $\Lambda$ respectively, such that $W^0_k = V_k$ and $\Xi^0 = \Lambda$, and such that, for $t< \smo$, $\Ls \subset \Lt$ and $\Vsp_k \subset W_k^t$. 
We assume that,  in this smoothness scale, $A^+ B^T$ is smoothness preserving up to a certain range $R$,  that is, that $\lambda \in \Ls$ implies that $A^+ B^T \lambda \in \Vsp = \prod_k \Vsp_k$, and 
\begin{equation}\label{eq:reg_stab}
\| A^+ \opB^T \lambda \|_{\Vsp} \leq C_r \| \lambda \|_{\Ls}, \qquad 0 \leq \smo \leq R.
\end{equation}

We make one of the two following assumptions on the behavior  of the black-box solver $A_h^+$ as the ``meshsize'' parameter $h$ tends to zero.
\begin{assmp}\label{consist_strong}
The discrete solver $\opA ^+_h$ satisfies the following estimate: $A^+ f \in \Vsp$, $0 \leq r \leq R$,
  implies 
\[
\| (\opA ^+_h - \opA ^+) f \|_{V} \lesssim h^{\smo} \| A^+ f \|_{\Vsp}.
\]
\end{assmp}

\begin{assmp}\label{consist_weak}
The discrete solver $\opA ^+_h$ satisfies the following estimate: $A^+ f \in \Vsp$, $0 \leq r \leq R$, implies
\[
\| \opB (\opA ^+_h - \opA ^+) f \|_{\Lambda'} \lesssim h^{\smo} \| A^+ f \|_{\Vsp}.
\]
\end{assmp}
Assumption \ref{consist_strong} ensures the individual convergence of approximations for all the subproblems, whereas the weaker Assumption \ref{consist_weak} only ascertains that the local solvers are sufficiently accurate to approximate the coupling. If only the latter holds, some other local solvers must be eventually used to reconstruct the full solution of the subproblems from the computed $\lambda_\delta$ and $z^*$. Note also that Assumption \ref{consist_strong} with $\smo=0$ reduces to the boundedness of $\opA ^+_h$.
\

\

We next  make the quite natural assumption that $\Lambda_\delta$ is an order $m$ approximation space.
\begin{assmp}\label{approx-lambda} For all $\lambda \in \Ls$, $\smo >  0$,
\[\inf_{\mh  \in \Lambda_\delta} \| \lambda - \mh  \|_{\Lambda} \lesssim \delta^{\min\{\smo,m\}} \| \lambda \|_{\Ls}.\]
\end{assmp}
The analysis below also requires the following inverse inequality to hold true in $\Lambda_\delta$, where for the sake of simplicity and without loss of generality we assume that $t \leq R$ ($R$ being the parameter in \eqref{eq:reg_stab}).
\begin{assmp}\label{ass:inverse}
	For some $t > 0$, $\Lambda_\delta \subset \Xi^t$, and for all $\lh  \in \Lambda_\delta$, 
\begin{equation}\label{inverse}
\| \lh  \|_{\Ls} \lesssim \delta^{-r} \|\lh  \|_{\Lambda}, \quad 0 < r \leq t.
\end{equation}
\end{assmp}

We have the following lemma.
\begin{lem}\label{coerc_sh} 
Under Assumptions \ref{consist_weak} and \ref{ass:inverse}, the bilinear form $s_h$ is continuous on $\Lambda \times \Lh$, and there exists a constant $C_\Lambda$ such that if $\hO/\hS < C_\Lambda$ then $s_h$ is  coercive on $\hLh = \Lh\cap \Zperp$.
\end{lem}

\begin{proof}
We start by observing  that {for $r$ with $0< r  < \min\{t,R\}$,  by \eqref{eq:reg_stab} and Assumptions \ref{consist_weak},}  for $\lambda \in \Lambda$ and $\mh \in \Lh$, we can write
	\begin{multline}\label{eq:cont_s_1}
		\langle \opB (\opA ^+ - \opA ^+_h) \opB ^T \lambda , \mh  \rangle \lesssim \| \opB  (\opA ^+ - \opA ^+_h) \opB ^T \lambda  \|_{\Lambda'}
		\| \mh  \|_{\Lambda}\\
		\lesssim  h^\smo \| \opA^+ \opB ^T \lambda  \|_{\Vsp} \| \mh  \|_{\Lambda} \lesssim h^{\smo} \| \lambda  \|_{\Ls} \| \mh  \|_{\Lambda}.
	\end{multline}
	By taking $\smo=0$ in \eqref{eq:cont_s_1} it follows that
	\[
	s_h(\lambda, \mh) = s(\lambda, \mh) + \langle \opB (\opA ^+_h - \opA ^+) \opB ^T \lambda , \mh  \rangle \lesssim \| \lambda  \|_{\Lambda} \| \mh  \|_{\Lambda},
	\]
	which proves the continuity.
	
	Now, note that if $\lambda = \lh \in \Lh$ in \eqref{eq:cont_s_1}, then by Assumption \ref{ass:inverse}
	we have, for some $r > 0$,
	\[
	 h^{\smo} \| \lh  \|_{\Ls} \| \mh  \|_{\Lambda} \lesssim h^{\smo} \delta^{-\smo} \| \lh  \|_{\Lambda} \| \mh  \|_{\Lambda},
	\]
that is, for some constant $C > 0$ we have
\[\langle \opB (\opA ^+ - \opA ^+_h) \opB ^T \lh , \lh  \rangle \leq C h^{\smo} \delta^{-\smo} \| \lh  \|_{\Lambda}^2.\]

Let now $\lh \in \hLh = \Lh \cap \Zperp$. As $\lh \in \Zperp$, we know, by Theorem \ref{thm:wellposedness_cont}, that $s(\lh,\lh) \geq \Cs \| \lh \|_{\Lambda}^2$ for some $\Cs>0$. We can  write
\begin{gather*}
s_h(\lh ,\lh ) = \langle \opB  \opA ^+_h \opB ^T \lh , \lh  \rangle  =
\langle \opB 
 \opA ^+ \opB ^T \lh  , \lh  \rangle - \langle\opB  (\opA ^+ - \opA ^+_h) \opB ^T \lh , \lh  \rangle = \\
s(\lh ,\lh ) - \langle \opB  (\opA ^+ - \opA ^+_h) \opB ^T \lh ,\lh  \rangle
\ge \Cs \| \lh  \|^2_{\Lambda} - \langle \opB (\opA ^+ - \opA ^+_h) \opB ^T \lh , \lh  \rangle.
\end{gather*}
We conclude that
\[ s_h(\lh ,\lh ) \geq \Cs \| \lh  \|^2_{\Lambda} - C h^{\smo} \delta^{-\smo} \| \lh  \|_{\Lambda}^2,\]
which, provided $h$ and $\delta$ are chosen in such a way that $h^{\smo} \delta^{-\smo} < \Cs/C$, implies that $s_h$ is coercive. 
\end{proof}

\begin{rem}
The condition appearing in Lemma \ref{coerc_sh} for the coercivity of the Schur complement is similar to that introduced in \cite{Bab72} for the discrete stability of the full mixed Lagrange multiplier problem. We also refer to \cite{Bre74} for a discussion on the well-posedness and approximation of Lagrange multiplier formulations.
\end{rem}

Finally, for the well posedness of the discrete problem we will need  an inf-sup condition allowing to control $Z$ with elements of $\Lambda_\delta$. As $Z$ is a small fixed space (independent of $\delta$) the following mild assumption will be sufficient to this aim. 
\begin{assmp}\label{infsup}
	There is a (small) finite dimensional subspace  $\Lambda_0 \subset \Lambda$ such that $\Lambda_0 \subseteq \cap_\delta \Lambda_\delta$  and such that 
	\[
	\inf_{z\in Z} \sup_{\mu\in \Lambda_0} \frac{b(z,\mu)}{\| z \|_{V}\| \mu\|_{\lambda}} \gtrsim 1.
	\]
\end{assmp}

Assumption \ref{infsup} combined with Lemma \ref{coerc_sh} ensure the well posedness of Problem \ref{approx-lambda}, which yields an optimal approximation to the solution of the original problem, as stated by the following theorem.

\begin{thm}\label{thm:error}  Under Assumptions \ref{consist_weak}, \ref{ass:inverse} and  \ref{infsup},  $u - z \in \Vsp$, $r \leq R$  implies
\[
\| \lambda - \lh  \|_{\Lambda} + \| z - z^* \|_{V} \lesssim \inf_{\substack{\xid \in \Lambda_\delta} \atop
	{\opB^T (\lambda - \xid) \perp Z} 
	} \| \lambda - \xid \|_{\Lambda} + h^{\smo} \| u - z \|_{\Vsp}.
\]
Moreover, if \ref{consist_strong} holds, setting $u_h = A^+_h(f + B^T \lh) + z^*$ we have that 
\[
\| u - u_h \|_{V}  \lesssim \inf_{{\xid \in \Lambda_\delta} \atop
	{\opB^T (\lambda - \xid) \perp Z} 
} \| \lambda - \xid \|_{\Lambda} + h^{\smo} \| u - z \|_{\Vsp}.
\]
\end{thm}

\begin{proof}

Let $\xid \in \Lambda_\delta$ be any element satisfying 
\begin{equation}\label{condxh}
b(z, \lambda  - \xid) = 0.
\end{equation}

We have
\begin{gather*}
 \| \lh  - \xid \|_{\Lambda}^2  \lesssim s_h(\lh  - \xid,\lh  - \xid) \lesssim
s_h( \lambda - \xid, \lh  - \xid)\\
+ [( s( \lambda, \lh  - \xid)   - \langle g, \lh  - \xid \rangle )-
( s_h( \lambda, \lh  - \xid) - \langle g_h ,\lh  - \xid \rangle )]
\end{gather*}
which easily yields
\[
 \| \lh  - \xid \|_{\Lambda}  \lesssim
 \| \lambda - \xid \|_{\Lambda} +
 \sup_{\mh  \in \Lambda_\delta} \frac {| [s( \lambda, \mh )  - \langle g, \mh  \rangle]            - [s_h( \lambda, \mh  ) -  \langle g_h ,\mh   \rangle] |}{\| \mh  \|_{\Lambda}}.
\]
We now need to estimate the consistency errors. We observe that
\[
 s( \lambda, \mh )  - \langle g, \mh  \rangle = b(\opA ^+(\opB ^T \lambda + f),\mh )
\]
and
\[
s_h(\lambda,\mh ) - \langle g_h ,\mh   \rangle = b(\opA _h^+(\opB ^T \lambda + f),\mh ).
\]
Then
\begin{gather*}
| [s( \lambda, \mh )  - \langle g, \mh  \rangle]            - [s_h( \lambda, \mh  ) -  \langle g_h ,\mh  \rangle] | =
| b ( (\opA ^+ - \opA ^+_h)(\opB ^T \lambda + f) , \mh  ) |  \\ \lesssim \| \opB (\opA ^+ - \opA ^+_h)(\opB ^T \lambda + f) \|_{\Lambda'} \| \mh  \|_{\Lambda}.
\end{gather*}
Assumption \ref{consist_weak} yields
\[
| [s( \lambda, \mh )  - \langle g, \mh   \rangle] - [s_h( \lambda, \mh  ) -  \langle g_h ,\mh   \rangle] |  \lesssim h^{\smo} \|\opA^ + (\opB ^T \lambda + f) \|_{\Vsp}.
\]

Let us now bound $\| z - z^* \|_V$. We have 
\[
\| z - z^* \|_{V} \lesssim \sup_{\mu \in \Lambda_0} \frac{b(z - z^*,\mu)}{\| \mu \|_{\Lambda}}.
\]
Now, as $\Lambda_0 \subset \Lambda_\delta$ we have
\begin{multline*}
b(z - z^*,\mu) = \langle g - g_h , \mu \rangle - s(\lambda,\mu) + s_h(\lh,\mu) \\ = 
\langle g - g_h , \mu \rangle - s(\lambda,\mu) + s_h(\lambda,\mu) + s_h(\lh - \lambda,\mu)
\\
\lesssim h^{\smo} \| A^+(B^T \lambda + f) \|_{\Vsp} \| \mu \|_{\Lambda} + \| \lh - \lambda \|_{\Lambda} \| \mu \|_{\Lambda}.
\end{multline*}

Finally, let us estimate the error on $u$.
We have
\begin{multline*}
 u - u_h = A^+(f + B^T\lambda) - A^+_h(f + B^T \lh) + z - z^* \\ = 
 (A^+ - A^+_h)(f + B^T\lambda)  + A^+_h B^T(\lambda - \lh) + z - z^*. 
 \end{multline*}
 Therefore
 \begin{multline*}
 \|u - u_h\|_V = \| (A^+ - A^+_h)(f + B^T\lambda)  + A^+_h B^T(\lambda - \lh) + z - z^*\|_V\\
 \leq \underbrace{\| (A^+ - A^+_h)(f + B^T\lambda)\|_V}_{I} +  \underbrace{\| A^+_h B^T(\lambda - \lh)\|_V}_{II} + \|z - z^*\|_V\\
 \lesssim h^{\smo} \| u - z \|_{\Vsp} + \| \lambda - \lh \|_{\Lambda} + \| z - z^* \|_{V},
\end{multline*}
where we used Assumption \ref{consist_strong} and \eqref{eq:u_def} for $I$ followed by the boundedness of $A^+_h$ and the bound $\|B^T(\lambda - \lh)\|_{V'} \lesssim \| \lambda - \lh \|_{\Lambda}$ for $II$.
\end{proof}

It remains to show that $\lambda$ can be optimally approximated by functions $\xid$ in $\Lambda_\delta$  with $\opB^T (\lambda - \xid) \perp Z$. This is the object of the following proposition.
\begin{prop}
	Under Assumption \ref{infsup} it holds that
	\begin{equation}\label{xidoptimal}
 \inf_{{\xid \in \Lambda_\delta} \atop
	{\opB^T (\lambda - \xid) \perp Z} 
} \| \lambda - \xid \|_\Lambda \lesssim  \inf_{{\xid \in \Lambda_\delta} 
}  \| \lambda - \xid \| _\Lambda
	\end{equation}

	\end{prop}

\begin{proof}
	It is sufficient to take $\xh$ as the solution to the problem
	\begin{eqnarray*}
		s(\xh ,\mh ) + b (z',\mh )  &=& s(\lambda,\mh) \qquad \forall \mh  \in \Lambda_\delta \\
		b(w, \xh ) &=& b(w,\lambda), \qquad \forall w \in Z.
	\end{eqnarray*}
	Standard error estimates for the solution of saddle point problems yield \eqref{xidoptimal}.
\end{proof}

\begin{cor}
	Under Assumptions \ref{consist_strong}, \ref{approx-lambda}, \ref{ass:inverse} and \ref{infsup}, if $u \in \Vsp$, $r \leq R$ and $\lambda \in \Xi^t$, $t \leq m$, setting $u_h = A^+_h(f + B^T \lh) + z^*$ we have that 
	\[
	\| u - u_h \|_{V}  \lesssim \delta^{t} \| \lambda \|_{\Xi^t} + h^{\smo} \| u - z \|_{\Vsp}.\]
	\end{cor}

\begin{rem}
Note that it is not the smoothness of $u$ that comes into play in Theorem \ref{thm:error}, but the smoothness of $u - z$.
\end{rem}

\begin{rem}\label{rem:noncoercive}
	Our framework requires the bilinear form $a$ to be coercive on $(\ker A_k)^\perp$. On the other hand, it often happens that $a$ is only semi definite on $(\ker A_k)^\perp$ (for instance when the local problems are themselves saddle point problems), satisfying inf-sup conditions of the form
	\[
	\inf_{v \in (\ker A_k)^\perp} 	\sup_{w \in (\ker A_k)^\perp} \frac{a_k(v,w)}{\| v \|_{V_k} \| w \|_{V_k}} \gtrsim 1, \qquad 	\inf_{w \in (\ker A_k)^\perp} 	\sup_{v \in (\ker A_k)^\perp} \frac{a_k(v,w)}{\| v \|_{V_k} \| w \|_{V_k}} \gtrsim 1.
	\]
	However, following the approach of \cite{brezzif1993stabilization}, it is always possible to write down an equivalent coercive form of the local problem falling in our framework, by adding a residual term measured in the $V_k'$ norm. This approach allows us to extend our analysis to more general situations. By treating local numerical solvers as black boxes, we avoid dependence on the specific formulation of the local problem. Consequently, these solvers can be designed according to the original (non-coercive) formulation.
\end{rem}

%
%
%
%

\subsection{Including stabilization}\label{sec:stab}

Consider a situation where the condition  $\hO/\hS< C_\Lambda$ as given by Lemma \ref{coerc_sh}, does not hold, but where we can not (or do not want to) reduce $\hO$ or increase $\hS$. Assume however that we know \eqref{inverse} holds for a coarser mesh-size $\thS$ where $ \hO/\thS< C_\Lambda$, for some {auxiliary} space $\tLh$ that is not necessarily related to $\Lambda_\delta$ but has similar asymptotic approximation properties. This can be leveraged to design a stabilization term and ensure the well-posedness of the discrete problem.

\

To this aim we let $\tpiS: \Lambda \to \tLhs$ denote a  bounded projection on $\tilde \Lambda_{\hS}$. Moreover we let $\dsc{\cdot}{\cdot}: (\Lh + \tLhs) \times (\Lh + \tLhs) \to \mathbb{R}$ denote a continuous bilinear form satisfying 
\[
\dsc{\lambda_\delta - \tpiS\lambda_\delta}{\lambda_\delta - \tpiS \lambda_\delta} \gtrsim \| \lambda_\delta - \tpiS \lambda_\delta \|^2_{\Lambda}, \qquad \text{ for all }\lambda_\delta \in \Lh,
\]
and we set
\[
j(\lh,\mh) = \dsc{\lh - \tpiS \lh}{\mh - \tpiS \mh}.
\]

We consider the following stabilized version of Problem \ref{feti_discrete}, where $\gamma \in \mathbb{R}^+$ is a stabilization parameter the size of which will influence the stability of the system.

\begin{pb}\label{feti_stabilized} Find $\lh  \in \Lambda_\delta$ and $z^* \in Z$
	such that
	\begin{eqnarray*}
		s_h(\lh ,\mh ) + \gamma j(\lh ,\mh)+ b (z^*,\mh )  &=& \langle g_h,\mh  \rangle, \qquad \forall \mh  \in \Lambda_\delta, \\
		b(w, \lh ) &=& -\langle f, w \rangle \qquad \forall w \in Z.
	\end{eqnarray*}
\end{pb}

We have the following theorem.

\begin{thm}\label{thm:error_stab}
Assume that the black box approximate solver $\opA_{h}^+$ satisfies Assumptions \ref{consist_weak} and \ref{infsup}, and that the space $\tLhs$ satisfies Assumption \ref{ass:inverse}. Then, there exists $\gamma_0$ such that, provided $\gamma > \gamma_0$, Problem \ref{feti_stabilized} is well posed, and the following error bound holds for its solution:
\[
\| \lambda - \lh  \|_{\Lambda} + \| z - z^* \|_{V} \lesssim \inf_{{\xid \in \Lambda_\delta} \atop
	{\opB^T (\lambda - \xid) \perp Z} 
} \| \lambda - \xid \|_{\Lambda}  + \inf_{\xi_{\tilde \delta} \in \tLhs} \| \lambda - \xi_{\tilde \delta}\|_{\Lambda}+  h^{\smo} \| u - z \|_{\Vsp} .
\]
Moreover, if Assumption \ref{consist_strong} holds, setting $u_h = A^+_h(f + B^T \lh) + z^*$ we have that 
\[
\| u - u_h \|_{V}  \lesssim \inf_{{\xid \in \Lambda_\delta} \atop
	{\opB^T (\lambda - \xid) \perp Z} 
} \| \lambda - \xid \|_{\Lambda}  + \inf_{\xi_{\tilde \delta} \in \tLhs} \| \lambda - \xi_{\tilde \delta}\|_{\Lambda} + h^{\smo} \| u - z \|_{\Vsp}.
\]	
	\end{thm}

\begin{proof} To prove well posedness we  follow the proof of Lemma \ref{coerc_sh}. 
For $\lh \in \Lh \cap \widehat \Lambda$	we can write
\[
s_h(\lh ,\lh ) = \left<B A_h^+ B^T \lh ,\lh \right> = \left<B A_h^+ B^T \tpiS \lh ,\lh \right> + \left<B A_h^+ B^T (I - \tpiS) \lh ,\lh \right>.
\]
In the first term we first add and subtract $- B A^+ B^T \tpiS \lh$  on the left slot to get
\begin{gather*}
	\left<B A_h^+ B^T \tpiS \lh ,\lh \right> = \left<B (A_h^+ - A) B^T \tpiS \lh ,\lh \right> + \left<B A^+ B^T \tpiS \lh ,\lh \right>,
\end{gather*}
and then, letting $\Cs $ be the coercivity constant of $s(\cdot,\cdot)$, we bound from below the  second term on the right hand side by adding and subtracting $- B A^+ B^T \lh $, which allows us to write 
\begin{multline*}
	\left<B A^+ B^T \tpiS \lh ,\lh \right> = \left<B A^+ B^T (\tpiS  - I)\lh ,\lh \right>  + s(\lh ,\lh ) \\
	\ge \left<B A^+ B^T (\tpiS  - I)\lh ,\lh \right>+ \Cs  \|\lh \|_{\Lambda}^2.
\end{multline*}
Now we can bound 
\begin{multline*}
	\left<B A^+ B^T (\tpiS  - I)\lh ,\lh \right>  \leq C \| \tpiS \lh - \lh \|_{\Lambda} \| \lh \|_{\Lambda} \leq C \sqrt{j(\lh,\lh)} \| \lh \|_{\Lambda}\\ \leq c_1 \epsilon^{-1} j(\lh,\lh) + \frac 1 4 \epsilon \| \lh \|_{\Lambda}^2.	
	\end{multline*}
Analogously we have that
\[
	\left<B A_h^+ B^T (\tpiS  - I)\lh ,\lh \right>  \leq c_2 \epsilon^{-1} j(\lh,\lh) + \frac 1 4 \epsilon \| \lh \|_{\Lambda}^2.	
\]

Using now the assumed properties of the black box approximate solver $A^+_h$ we have that
\begin{multline*}
	\left<B (A_h^+ - A) B^T \tpiS \lh ,\lh \right> \leq C \hO^{\smo} \|\tpiS \lh \|_{\Ls} \|\lh \|_{\Lambda} 
	\leq \frac {\hO^{\smo} }{\thS^{\smo}} \|  \tpiS \lh \|_V \|\lh \|_{\Lambda} 
	\leq c_3 \frac {\hO^{\smo} }{\thS^{\smo}}  \|\lh \|_{\Lambda}^2.
\end{multline*}
Collecting terms we see that
\[
s_h(\lh ,\lh ) \ge \left(\Cs  - \frac12 \epsilon - c_3 \left(\frac{\hO}{\thS} \right)^{\smo}\right) \|\lh \|_{\Lambda}^2 - 2  (c_1+c_2)\epsilon^{-1} j(\lh ,\lh ).
\]
Assuming that $\epsilon = 1/2 \Cs $ and 
\[
c_3 \left(\frac{\hO}{\thS} \right)^{\smo} \leq \frac14 \Cs 
\]
the partial coercivity follows 
\begin{equation}\label{partialcoercivity}
s_h(\lh ,\lh ) \ge \frac12 \Cs  \|\lh \|_{\Lambda}^2 - 4 \Cs  (c_1+c_2) j(\lh ,\lh ).
\end{equation}
Considering this bound in the stability estimate of Problem \ref{feti_stabilized}  we see that
\[
s_h(\lh ,\lh ) + \gamma j(\lh ,\lh )  \ge \frac12 \Cs  \|\lh \|_{\Lambda}^2 + (\gamma- 4 \Cs  (c_1+c_2)) j(\lh ,\lh )
\]
and hence
the stabilized method has similar stability properties as the inf-sup stable method if $\gamma \ge 4 \Cs  (c_1+c_2)$.

\

We conclude that for such a choice of $\gamma$,
\begin{equation}\label{eq:coerciv}
	\frac12 \Cs  \|\lh \|_{\Lambda}^2 \leq s_h(\lh ,\lh ) + \gamma j(\lh ,\lh ).
\end{equation}

\

For the error analysis we proceed as in Theorem \ref{thm:error}, using now the bound \eqref{eq:coerciv}. Let once again $\xh \in \Lh$ satisfy 
\eqref{condxh}, and let $e_h = \lh - \xh$, with now $\lh$ solution to Problem \ref{feti_stabilized}. We have
\[
\frac12 \Cs \|e_h \|_{\Lambda}^2 \leq s_h(e_h,e_h) + \gamma j(e_h,e_h).
\]
We see that
\[
s_h(e_h,e_h) + \gamma j(e_h,e_h) = 	\sh(\lambda - \xh, e_h) 
+ \langle g_h  - g, \eh \rangle 
+ s(\lambda,\eh) - \sh(\lambda,\eh) 
- \gamma j (\xh, e_h).
\]
The first four terms are bounded as in the proof of Theorem \ref{thm:error}. 
As far as the last term is concerned, we observe that
\[
j(\xh, e_h) \lesssim \| \xh - \tpiS \xh \|_{\Lambda} \| e_h \|_{\Lambda}.
\]
Now, adding and subtracting $(I - \tpiS)\lambda$ we see that
\[
\| \xh - \tpiS  \lambda \|_{\Lambda}  \leq \| \lambda - \tpiS \lambda \|_{\Lambda} + 
\|(I-\tpiS) (\lambda -  \xh) \|_{\Lambda} \lesssim
 \| \lambda - \tpiS \lambda \|_{\Lambda} + 
\| \lambda -  \xh \|_{\Lambda}.
\]
Then, proceeding as in the proof of Theorem \ref{thm:error}, we obtain the bound on $\| \lambda - \lh \|_{\Lambda}$. 
 The bounds on $\| z - z^* \|_{V}$ and on $\| u - u_h \|_{V}$ also follow using the argument of Theorem \ref{thm:error}.
 \end{proof}

\NOTE{ Error estimate with all the steps. Let $\xh \in \Lh $ such that $b(w,\lambda - \xh) = 0$ for all $w \in Z$.
		 arbitrary. We set $e_h = \lh - \xh$. We have
	\begin{multline}
		\| \lh - \xh \|_{\Lambda}^2 \lesssim s_h(\lh-\xh,\lh-\xh) + \gamma j(\lh - \xh,\lh - \xh) \\
		= \sh(\lambda - \xh, e_h) + \sh(\lh - \lambda,\eh) 
		\pm s(\lambda,\eh)  + \gamma j (\lambda - \xh,\eh) + \gamma j (\lh - \lambda, \eh) \\
		=
		\sh(\lambda - \xh, e_h) 
		+ \langle g_h , \eh \rangle 
		- \sh(\lambda,\eh) 
		+ s(\lambda,\eh) 
	- \langle g , \eh \rangle 
		- \gamma j (\xh, e_h) \\
		\lesssim \| \lambda - \xh \|_{\Lambda} \| \eh \|_{\Lambda} + \| g_h - g \|_{\Lambda'}  \| \eh \|_{\Lambda} + \| B (A^+ - A^+_h B^T) \lambda \|_{\Lambda'} \| \eh \|_{\Lambda} +  | j(\xi_\delta,\eh)|
	\end{multline}
	Now we have
	\[
	| j (\xh,\eh) | = | \dsc{\xh - \tpiS \xh}{\eh - \tpiS \eh}| \lesssim \| \xh - \tpiS \xh \|_\Lambda \| \eh \|_\Lambda 
	\]
	Adding and subtracting $\lambda - \tpiS\lambda$ we have
\[
\| \xh - \tpiS \xh \|_\Lambda \leq \| \xh - \lambda \|_\Lambda + \| \lambda - \tpiS \lambda \|_\Lambda 
\]
and then
\[
|j(\xh,\eh) \lesssim (\| \xh - \lambda \|_{\Lambda} + \| \lambda - \tpiS \lambda \|_{\Lambda}) \| e_h \|_{\Lambda}.
\]
Then we can write
\begin{gather*}
		\| \lh - \xh \|_{\Lambda}^2 \lesssim 
		\| \lambda - \xh \|_{\Lambda} \| \eh \|_{\Lambda} + \| g_h - g \|_{\Lambda'}  \| \eh \|_{\Lambda} \\
		+ \| B (A^+ - A^+_h B^T) \lambda \|_{\Lambda'}    \| \eh \|_{\Lambda}
		+   \| \xh - \lambda \|_{\Lambda} + \| \lambda - \tpiS \lambda \|_{\Lambda}   \| \eh \|_{\Lambda}
\end{gather*}
whence
\begin{gather*}
	\| \lh - \xh \|_{\Lambda} \lesssim 
	\| \lambda - \xh \|_{\Lambda} + \| g_h - g \|_{\Lambda'}  + \| B (A^+ - A^+_h B^T) \lambda \|_{\Lambda'} + \| \lambda - \tpiS \lambda \|_{\Lambda}.
\end{gather*}
}

The lower bound $\gamma > \gamma_0$ on the stabilization parameter, required by Theorem \ref{thm:error_stab}, can be relaxed if $A_h^+$ is a positive semi definite operator, so that $s_h(\lh,\lh) \ge 0$ for all $\lh \in \Lh$. More precisely, we have the following corollary.
\begin{cor}\label{cor:3.13}
	Under the assumptions of Theorem \ref{thm:error_stab}, if  $s_h$ is positive semi-definite then the error bounds hold for all $\gamma>0$, with a hidden constant that scales as $\gamma^{-1}$.
\end{cor}
\begin{proof}
	We only need to show that coercivity holds for all $\gamma>0$. To this end we observe that, as $s_h(\lh,\lh) \geq 0$,  for $\eta \in (0,1)$ arbitrary we have that
	\[
	s_h(\lh,\lh) + \gamma j(\lh,\lh)  \ge \eta s_h(\lh,\lh) + \gamma j(\lh,\lh).
	\]
Using \eqref{partialcoercivity} to bound the first term of the right hand side we get
	\[
	s_h(\lh,\lh) + \gamma j(\lh,\lh) \ge \frac{\eta}{2} \Cs  \|\lh \|_{\Lambda}^2 + (\gamma - 4 \eta \Cs  (c_1+c_2)) j(\lh ,\lh ).
	\]
	We obtain  coercivity by choosing $\eta < \gamma/(4 \Cs  ((c_1+c_2))$ and the claim follows.
\end{proof}

Corollary \ref{cor:3.13} applies, in particular, to the case where, mimicking \eqref{eq:pseudo_inverse_def} on the discrete level, for all $k$, $A_{h,k}^+$ is defined through a Galerkin projection on a finite dimensional subspace $C_{h,k} \subset \carrierAk$. Indeed, for $g \in V_k'$, we can define $\opA ^+_{h,k} g \in C_{h,k}$ as the solution to
\begin{equation}\label{eq:pseudo_inverse_disc}
a_k(\opA ^+_{h,k} g , w_h) = \langle g , w_h \rangle, \qquad \text{ for all } w_h \in C_{h,k}.
\end{equation}
Then, to prove the positive semi definiteness of $s_h$, using \eqref{eq:pseudo_inverse_disc} and the coercivity of $a_k$ on $\carrierAk$ we can write, for all $\lh\in \Lh$,
\begin{equation*}
	0 \leq a_k (\opA^+_{h,k} B^T \lh, \opA^+_{h,k} B^T \lh ) = \langle B^T \lh, \opA^+_{h,k} B^T \lh \rangle \\= \langle  \lh, B \opA^+_{h,k} B^T \lh \rangle = s_h(\lh.\lh).
	\end{equation*}
%
%

\begin{rem}
We note that since $\Lambda_\delta$ is finite dimensional, it is often possible to choose $j(\xh,\xh)$ on a form convenient for computation. Indeed in some cases the projection operator can be eliminated and replaced by some other operator acting directly on $\xh$. This is the case for instance when, in the domain decomposition framework (see Section \ref{dd}), $\tLh$ and $\Lambda_\delta$ are different finite element spaces of different mesh size, but with the same polynomial order. Then typically $\dsc{\xh - \tpiS \xh}{\xh - \tpiS \xh}$ can be replaced by a penalty term acting on the jumps of derivatives of $\xh$ alone, see \cite{Bu14}.  In such a framework, computable bilinear forms $[\cdot,\cdot]_\delta$ can also be designed by resorting to suitable localization results, see \cite{bertoluzza2023localization}.  For an alternative way of stabilizing in an abstract framework, see also \cite{bertoluzza2021algebraic}.
\end{rem}
\newcommand{\dlh}{\eta_\delta}

\


\section{Solving the resulting discrete problem}\label{sec:feti}
Since we are interested in large problems, it is natural to consider an iterative solver for the reduced system of Problem \ref{feti_discrete}, where the stiffness matrix relative to the approximate Schur complement operator needs not be assembled, but where, at each iteration, the approximate black box solvers are called upon to evaluate their action on a given element of $\Lh$.  We note that for very large systems it may be advantageous to proceed in a nested fashion, where also the subproblems are themselves decomposed in smaller subproblems, until  the local solvers are so small that direct solvers can be applied. 
Focusing on the unstabilized Problem \ref{feti_discrete}, and
following once again an approach inspired by the FETI method, in order to efficiently solve Problem \ref{feti_discrete} we will need to
\begin{enumerate}[a)]
	\item reduce the saddle point problem posed on $\Lambda_\delta \times Z$ to a positive definite problem, defined on the subspace $\hLh$ of $\Lh$, which can be solved by a Krylov type method (e.g. PCG, if the corresponding operator is symmetric);
	\item construct a preconditioner for the the restriction of the bilinear form $s_h$ to the  subspace $\hLh = \Lh \cap \widehat\Lambda$.
\end{enumerate}

To this aim, we will leverage a number of linear operators that we will have to construct either directly or through the corresponding bilinear form.

\subsection*{Reduction to $\hLh$}  In order to reduce the problem to the subspace $\hLh$, we start by 
constructing a mapping $\Embed: \hLh' \to \Lh'$ which will also be instrumental in the definition of the preconditioner.  To this aim, letting  $\Gd: Z \to \Lh'$ be defined, with a slight abuse of notation, as 
\begin{equation}\label{defG2}
	\langle \Gd z , \lambda \rangle = \langle \Bd z, \lambda \rangle = b(z,\lambda), \ \forall z\in Z, \ \lambda \in \Lh,
\end{equation}we start by choosing an easily computable discrete symmetric positive semi definite bilinear form $\tsigma: \Lh' \times \Lh' \to \mathbb{R}$, which we assume to be continuous on $\Lh'$ and coercive on $\Range{\opG}$: for all $z \in Z$, $\phi,\psi \in \Lh'$
\begin{equation}\label{boundsigmastar}
	\sigma^*(G z,G z) \gtrsim \| Gz \|_{\Lambda'}^2, \qquad \sigma^*(\phi,\psi) \lesssim \| \phi \|_{\Lambda'} \| \psi \|_{\Lambda'}.
\end{equation}
The bilinear form $\sigma^*$ induces a scalar product on $\Range\Gd$, which allows to define a projector $\Qsigma: \Lh' \to \Range\Gd$ (easily computable) as
\begin{equation}\label{defQsigma}
	\Qsigma \phi \in \Range{\opG}\qquad \text{ solution of }\qquad \tsigma(\Qsigma \phi,\psi) = \tsigma(\phi,\psi)\qquad \forall \psi \in \Range\opG. 
\end{equation}
\begin{rem}
Choosing $\sigma^*$ equal to the $\Lambda'$ scalar product automatically yields \eqref{boundsigmastar}. In principle, we could then use such a scalar product in the place of the bilinear form $\sigma^*$. $\Qsigma$ would then be the $\Lambda'$ orthogonal projection onto $\Range G$.
However, in our framework, such a choice would  often results in  $\Qsigma$ being quite difficult to compute. This is the reason why a simpler bilinear form $\sigma^*$ needs to be introduced.
\end{rem}
Let now $\tphi \in \hLh'$. 
We define $\Embed\tphi  \in \Lh'$ by setting
\begin{equation}\label{defK}
	\Embed \tphi = \phi - \Qsigma \phi,\end{equation}
where
$\phi \in \Lh'$ is any element such that 
\[
\langle \tphi , \hlambda \rangle = \langle \phi , \hlambda \rangle, \qquad \text{ for all }\hlambda \in \hLh.
\]

It is not difficult to check that if $\langle \phi - \phi' , \hlambda \rangle = 0$ for all $\hlambda \in \hLh$, then $\phi - \Qsigma \phi = \phi' - \Qsigma\phi'$, so that $\Embed$ is well defined, independently on how $\phi$ is chosen.
\begin{rem}\label{rem:Embed}
If we think of  $\hLh$ and $\Lh$ as independent spaces, and
 introduce an operator $\Inject: \hLh \to \Lh$, representing the natural injection, we can identify $\hLh'$ with the quotient space $\Lh'/\ker \Restrict$ where $\Restrict = \Inject^T : \Lh' \to \hLh'$ is the natural restriction operator, defined as
	\[
	\langle \Restrict \phi, \hlambda \rangle = \langle \phi , \Inject \hlambda \rangle \qquad \phi \in \Lh',\ \hlambda \in \hLh. 
	\] 
It is possible to show that the equivalence class $[\phi]$ of an element $\phi \in \Lh'$ can be identified with $\phi + \Range\Gd$.	The operator $\Embed: \hLh' \to \Lh'$ can  be interpreted as the selection of a specific ``canonical'' representative in $\Lh'$ of an equivalence class $\tphi = [\phi] \simeq \phi + \Range{G}$ in $\hLh'$. Remark  that $\Embed$ can be naturally extended to a mapping $\Embed : \Lh' \to \Lh'$ that, given any element $\phi$ of $\Lh'$ selects the canonical representative of the equivalence class $[\phi] = \phi + \Range{G}$.
	\end{rem}

\NOTE{
	I know that $\Lh' = \Range \opG \oplus \check \Lambda_\delta$ with $\check \Lambda_\delta = \Range {\Id{\Lh'} - \Qsigma}$. We can see that this is a direct sum, Indeed if $\phi \in  \Range \opG \cap  \Range {\Id{\Lh'} - \Qsigma}$ then there exists   $\phi \in \Lh'$ such that
	\[
	\phi = \Qsigma \phi = \phi' - \Qsigma \phi'.
	\]
	Then $\phi' - \phi = \Qsigma \phi'$ which implies that $\phi' \in \Range{\Qsigma}$ and then $\Id{\Lh'} \phi'- \Qsigma \phi' = 0$. 
	Corresponding to this splitting, I have a splitting of the dual space $\Lh$ as $\tilde\Lambda_\delta \oplus \ker G^T$, as depicted by the following scheme:
	\[\begin{array}{ccccc}
		\Range{\opG} & &\oplus& &\check\Lambda\\
		&\searrow & &\swarrow& \\
		\updownarrow & & \perp & & \updownarrow \\
		& \swarrow & & \searrow & \\
		\tilde \Lambda_\delta& &\oplus & &\hLh
	\end{array}
	\]
	It is easy to check that $\Range G = (\ker G^T)^\perp$. Indeed, we have
	\begin{multline*}
		\{ \hlambda \in \Lh: \hlambda  \perp \Range{G}\} =
		\{ \hlambda  \in \Lh: \langle Gz,\hlambda  \rangle = 0, \forall z \in Z\} 
		\\  =   \{ \hlambda  \in \Lh: \langle z, G^T \hlambda \rangle = 0, \forall z \in Z\} = \ker G^T.
	\end{multline*}
	I have $\Qsigma = G (G^T \Sigma G)^{-1} G^T \Sigma$. $\Qsigma: \Lh'\to \Lh'$ and $\QsigmaT  : \Lh \to \Lh$. I have that $(\Id{\Lh'} - \Qsigma)^T$ is the projection onto $\hLh$ orthogonally to $\Range{\opG}$. Check that this is a projection. $\lh \in \hLh = \ker G^T$. $\lh \in \ker G^T$, $\phi \in \Lh'$
	\[
	\langle \phi , \lh - \QsigmaT  \lh \rangle = \langle \phi, \lh \rangle - \langle \Qsigma \phi, \lh \rangle =
	\langle \phi, \lh \rangle - \langle G z, \lh \rangle =  \langle \phi, \lh \rangle
	\]
	implying that $\lh  - \QsigmaT  \lh = \lh$ for all $\lh \in \hLh$. Need to show that for all $\lh \in \Lh$, $\lh - \QsigmaT  \lh \in \ker G^T$. 
	\[
	\langle G^T (\lh - \QsigmaT  \lh) , z \rangle = 
	\langle \lh - \QsigmaT  \lh , G z \rangle = 
	\langle \lh , G z \rangle  - \langle   \lh , \Qsigma G z \rangle = \langle \lh , G z \rangle  - \langle   \lh , G z \rangle = 0
	\]
	\
	\[
	\Lh = \tilde \Lambda_\delta \oplus \hLh.
	\]
	with
	\[
	\hLh = \{ \lambda \in \Lh: \langle \phi ,\lambda \rangle = 0, \forall \phi \in \Range{G}\}, \quad 
	\tilde \Lambda_\delta = \{ \lambda \in \Lh: \langle \phi ,\lambda \rangle = 0, \forall \phi \in \check \Lambda_\delta\}.
	\]
	We have that $(\hLh, \check \Lambda_\delta)$ and $(\tilde \Lambda_\delta ,\Range{G})$ form two biorthogonal couples. 
}

\begin{rem}\label{rem:buildsigmastar} One possible way of constructing a bilinear form  $\sigma^*$ with the desired properties is by duality with respect to a scalar product on the space $\Lambda_0$ given by Assumption \ref{infsup}. Indeed, let $\sigma: \Lambda_0 \times \Lambda_0 \to \mathbb{R}$ be a scalar product inducing  on $\Lambda_0$ a norm equivalent to $\| \cdot \|_{\Lambda}$.
More precisely, assume that for all $\lambda,\mu \in \Lambda_0$, $\sigma$ satisfies
\begin{equation}\label{condsigma}
	\sigma(\lambda,\lambda) \gtrsim \| \lambda \|^2_{\Lambda}, \qquad \sigma(\lambda,\mu) \lesssim \| \lambda \|_{\Lambda} \| \mu \|_{\Lambda}.
\end{equation}
Letting $\{\eta_i,\ i=1,\cdots,N\}$ be a basis for $\Lambda_0$, $\Sigma = (\sigma_{ij})_{ij}$, with $\sigma_{ij} = \sigma(\eta_j,\eta_i)$, the  stiffness  matrix induced by $\sigma$, and $\Sigma^{-1} = (\sigma^*_{ij})_{ij}$ its inverse, we can define $\sigma^*$ as
\[
\sigma^*(\phi,\psi) = \sum_{i,j=1}^N \sigma^*_{ij} \langle \phi,\eta_j \rangle \langle \psi,\eta_i \rangle.
\]
By the arguments in \cite{bertoluzza2021algebraic} it is not difficult to check that $\sigma^*$ is a scalar product on $\Range{G}$ inducing a norm equivalent to the $\Lambda'$ norm, and that the corresponding projector $Q^*_\sigma$ satisfies
\[
\| Q^*_\sigma \phi \|_{\Lambda'} \lesssim \| \phi \|_{\Lambda'},
\]
the implicit constant in the inequality only depending on the coercivity and continuity constants implicit in \eqref{condsigma}.
\end{rem}

\NOTE{I have $\Lambda_0$ with 
	\[
	\inf_{z\in Z} \sup_{\lambda\in \Lambda_0} \frac{b(z,\lambda)}{\| \lambda \|_\Lambda
		\| z \|_V
	}  = \inf_{z} \sup_{\lambda\in \Lambda_0} \frac{\langle Gz,\lambda \rangle}{\| \lambda \|_\Lambda
		\| z \|_V
	}   \gtrsim 1
	\] 
	We have a basis for $\Lambda_0$ $\{\eta_i, i=1,\cdots,N\}$.  I have a dual basis $\{\zeta_i, 1=1,\ldots,N\}$ for $\tilde\Lambda_0\subseteq \Lambda_\delta'$, with \[
	G z = \sum_i b(z,\eta_i) \zeta_i, 
	\]
	and with
	\[
	\inf_{\lambda \in \Lambda_0}\sup_{\zeta \in\tilde \Lambda_0} \frac{\langle\zeta,\lambda\rangle}{\| \zeta\|_{\Lambda'}\| \lambda \|_{\Lambda}} \gtrsim 1.
	\]
	This implies that the projection onto $\Lambda_0$ orthogonally to $\tilde\Lambda_0$, as well as its adjoint are bounded
	\[
	\| \Pi_0 \lambda \|_\Lambda \lesssim \| \lambda \|_{\Lambda}, \qquad 	\| \Pi_0 \lambda \|_{\Lambda'}\lesssim \| \lambda \|_{\Lambda'},
	\]
	The two bases for $\Lambda_0$ and $\tilde \Lambda_0$ can be completed in two biorthogonal  bases for $\Lh$ $\{\eta_i, i=1,\cdots,M\}$, with $\{\eta_i, i=N+1,\cdots,M\}$ orthogonal to $\Lambda'_0$ and $\{\zeta_i, i=1,\cdots,M\}$, with $\{\zeta_i, i=N+1,\cdots,M\}$ orthogonal to $\Lambda_0$.
	Assume that I define a semi-scalar product $\sigma: \Lambda_0 \times \Lambda_0 \to \mathbb{R}$. Let $\sigma: \Lambda_\delta \times \Lambda_\delta$ be defined as
	\[
	\sigma(\lambda,\mu) = \sigma(\Pi_0 \lambda,\Pi_0 \mu), \qquad \Pi_0\lambda = \sum_{i=1}^N \langle \lambda ,\zeta_i \rangle \eta_i.
	\]
	Now I can define the dual semi-scalar product
	as 
	\[
	\sigma^*(\phi,\psi) = \sum_i \sum_j \Sigma^*_{ij} \langle \psi,\eta_i \rangle \langle \phi, \eta_j \rangle, \qquad \Sigma^* = \Sigma^{-1}, \quad \Sigma = (\sigma_{ij})_{ij} = (\sigma(\eta_i,\eta_j))_{ij}
	\]
	What I need for everything to work is that the projection obtained with this scalar product is bounded. I observe that letting $\Pi^*_0: \Lambda_\delta' \to \tilde\Lambda_0 < \zeta_i, i=1,\cdot,N >$ denote the adjoint of $\Pi_0$, I have 
	\[
	Q^*_\sigma \phi = Q^*_\sigma \Pi^*_0 \phi  
	\] 
	
	We know that it is bounded as follows
	\[
	\sup_{\lambda \in \Lambda_0} \frac{\langle Q_\sigma^* \phi, \lambda \rangle }{\| \lambda \|_{\Lambda}} \lesssim \| \Pi^*_0  \phi \|_{\Lambda'} \lesssim \| \phi \|_{\Lambda'}
	\]
	
}

We let $\Pi_\sigma : \Lh \to \Lh$ be defined as 
\begin{equation}\label{defPisigma} \Pi_\sigma = (\Id{\Lh'} - \Qsigma)^T = \Id{\Lh} - \QsigmaT. \end{equation}
 It is not difficult to check that $\Range{\Pi_\sigma}\subseteq \hLh$ and that $\lh \in \hLh$ implies that $\Pi_\sigma \lh = \lh$. Then $\Pi_\sigma$ is a projection onto $\hLh$.
Letting $\lambda$ be the solution to Problem \ref{feti_discrete}, we can see that $\lambda^0 = \lambda - \Pi_\sigma \lambda$ can be computed as
\[
\lambda^0 = \QsigmaT  G z^0, \quad \text{ with $z^0 \in Z$ solution of }\quad \sigma^*(G z^0, G z) = - \langle f, z \rangle,\ \forall z \in Z,
\]
\NOTE{In this note $\Sigma$ is an old notation, standing for the matrix corresponding to the semi-scalar product $\sigma^*$ on all $\Lh'$. 
$G^T \Sigma G: Z \to Z'$ invertible, by the coercivity of $\Sigma$ on $\Range\opG$. I know $G^T \lambda^0 = -f$. I can write
\begin{gather}
G^T \lambda^0 = -f \qquad \text{multiply by $\Sigma G (G^T \Sigma G)^{-1}$}\\
\Sigma G (G^T \Sigma G)^{-1} G^T \lambda^0 = - \Sigma G (G^T \Sigma G)^{-1} f, \qquad \text{write $f$ as $(G^T \Sigma G) z_f$}\\
\underbrace{\Sigma G (G^T \Sigma G)^{-1} G^T}_{\QsigmaT } \lambda^0 = - \underbrace{\Sigma G (G^T \Sigma G)^{-1}(G^T}_{\QsigmaT } \Sigma G) z_f, \qquad \text{isolate $\Qsigma$}\\
\QsigmaT  \lambda^0 = - \QsigmaT  \Sigma G z_f, \qquad \text{use $\QsigmaT  \lambda^0 = \lambda^0$,}
\end{gather}
}
and that $\hlambda = \Pi_\sigma \lambda = \lambda - \lambda^0$ is the solution to the following reduced problem
\begin{equation}\label{feti_reduced}
	s_h ( \hlambda , \hmu ) = \langle  g, \hmu \rangle - \langle  S_h \lambda^0,\hmu \rangle, \qquad \forall \hmu \in \hLh.
\end{equation}

We know by Lemma \ref{coerc_sh} that $s_h$ is coercive on $\hLh$, and, consequently, that the corresponding operator $\hSh: \hLh \to \hLh'$ is invertible. The reduced problem \eqref{feti_reduced} can be then addressed by any iterative solver well suited to solve linear systems with positive definite matrices. For instance, if, in addition, the operator $A^+_h$ is symmetric, which implies that $S_h$ and, consequently, $\hSh$ are symmetric, one can resort to the preconditioned conjugate gradient method.

\newcommand{\M}{M}
\newcommand{\m}{m}

\subsection{Construction of the preconditioner $\hM^{-1}$}
 A preconditioner $\hM^{-1}: \hLh' \to \hLh$ will be constructed by combining the mapping $\Embed$ with a suitably defined  preconditioning operator $\M^+: \Lh' \to \Lh$, corresponding to a bilinear form $\m^+: \Lh'\times \Lh' \to \mathbb{R}$.
 \label{sec:constructM} We will build $\hM^{-1}$ by building the corresponding bilinear form $\m^*: \hLh'\times \hLh' \to \mathbb{R}$ such that
\[
\langle \hM^{-1} \tphi,\tpsi \rangle = \m^*(\tphi,\tpsi), \qquad \tphi, \tpsi \in \hLh'.
\]
The bilinear form $\m^*$ will have the form 
\[
\m^*(\tphi,\tpsi) = \m^+(\Embed \tphi,\Embed \tpsi)
\]
where $\m^+: \Lh'\times\Lh' \to \mathbb{R}$, which is the bilinear form that we we will actually construct, corresponds to a linear operator $\M^+: \Lh'\to \Lh$ ``mimicking''
 some kind of inverse of $S_h$, if $S_h$. If $A^+_h$ and $\Bd$ were invertible, this would be
\begin{equation}\label{formalinversion}(B A^+_h B^T)^{-1} = B^{-T} (A^+_h)^{-1} B^{-1}.\end{equation}
Of course, none of the above mentioned operators is generally invertible; the idea is then to use the right-hand side expression in \eqref{formalinversion} as a guideline, and replace $B^{-1}$ and $(A_h^+)^{-1}$, respectively, with a suitable pseudo inverse $\Bd^+$, and with $A$ (or some spectrally equivalent $\widetilde A$).

Following the approach put forward by the FETI preconditioner, we rely on three ingredients.
The first ingredient is a finite dimensional space $\Vh \subseteq V$, to which the image of $\Bd^+$ will belong. We assume $\Vh$ to satisfy 
\begin{equation}\label{infsupprecond}
	Z \subseteq \Vh, \qquad \text{and}\qquad \inf_{\lambda \in \Lh} \sup_{v \in \Vh} \frac{b(v,\lambda)} {\| v \|_{V} \| \lambda \|_\Lambda} \gtrsim 1.
\end{equation}
Observe that, if a space $\Wh$ underlying the approximate pseudo inverse $A^+_h$  is available, a natural choice for $\Vh$ is $\Vh = \Wh$, but, in our framework, we are also interested in situations in which such a space might not be accessible. {Also note that $\Vh$ needs not satisfy any approximation assumption, so that, depending on the framework, its dimension can be sensibly smaller that the dimension of $V_h$. }
We let $\Ah : \Vh \to \Vh'$ and $\Bd : \Vh \to \Lh'$, be the discrete versions of the operators $A$  and $B$, defined as
\[
\langle \Ah v, w \rangle = a(v,w) \ \forall v,w \in \Vh, \qquad \langle \Bd v, \lambda \rangle = b(v,\lambda) \ \forall v \in \Vh, \lambda \in \Lh.
\]

The second ingredient is a discrete positive semi definite bilinear form $d: \Vh  \times \Vh  \to \mathbb{R}$ such that $d$ is coercive on $\ker \Bd$, and a corresponding splitting of the space $\Vh $ as
\[\Vh  = \ker \Bd \oplus (\ker \Bd)^\perp,\]
where
\begin{gather*}
	\ker \Bd = \{ v \in \Vh: \ b(v,\mu) = 0, \ \forall \mu \in \Lh\},\\
	(\ker \Bd)^\perp = \{ v \in \Vh: \ d(v,w) = 0, \ \forall w \in \ker \Bd\}.
\end{gather*}
We let $\Qd: \Vh  \to \ker \Bd$  denote the $d$-orthogonal projection onto $\ker \Bd$, and $\Pi_d : \Vh \to (\ker \Bd)^\perp$ be defined as $\Pi_d = \Id{\Vh} - \Qd$. We observe that for all $w \in \Vh$ we have that
\begin{equation}\label{dboundedness}
	d(\Qd w,\Qd w) \leq d(w,w), \qquad d(\Pi_d w,\Pi_d w) \leq d(w,w).
\end{equation}
Whichever the choice of the bilinear form $d$, the restriction of $\Bd$ to $(\ker \Bd)^\perp$ is injective and, thanks to the inf-sup condition \eqref{infsupprecond}, surjective. We then let $\Bd^+:  \Lh' \to (\ker \Bd)^\perp$ denote its inverse, which is a pseudo inverse of $\Bd$, satisfying
\begin{equation}\label{propBdplus}
	\Bd \Bd^+ = \Id{\Lh'}, \qquad \Bd^+ \Bd = \Pi_d.
\end{equation}
We easily realize that, given $\phi \in \Lh'$, the pseudo inverse $\Bd^+ \phi$ can be defined as $\Bd^+ \phi = v$, where $(v,\zeta) \in \Vh  \times \Lh$ is the uniquely defined solution to
\begin{equation}\label{charBplus}
	\begin{array}{lcl}
		d (v , w ) + b(w ,\zeta) &=& 0,\qquad \forall w \in \Vh\\
		b (v,\mu ) &=& \langle \phi, \mu \rangle,\qquad \forall \mu \in \Lh.
	\end{array}
\end{equation}

\newcommand{\snormta}[1]{| #1 |_{\ta}}

The third ingredient is a set of local preconditioners for $\Ah$: for all $k$ we let $\ta^k: \Vh^k \times \Vh^k \to \mathbb{R}$ denote a  coercive bilinear form satisfying, for some constants $K^* \geq 1 \geq \kappa_* > 0$, 
\begin{equation}\label{boundatilde}
	\kappa_* a^k(v,v) \leq \ta^k(v,v) \leq K^* a^k(v,v)\qquad \forall v \in \Vh^k.
\end{equation} 
We let  $\tAh^k$  denote the corresponding  linear operator, and 
we let $\ta:\Vh \times \Vh \to \mathbb{R}$ and $\tAh: \Vh \to \Vh'$ denote the related product bilinear form and corresponding linear operator. We let $| \cdot |_{\ta}$ denote the seminorm induced on $\Vh$ by the bilinear form $\ta$:
\begin{equation}\label{defnormatilda}\snormta{v} = \sqrt{\ta(v,v)}.\end{equation}

\

\newcommand{\tAd}{\widetilde A_\delta}

Following the guidelines provided by the formal inversion expression \eqref{formalinversion} we introduce the linear operator $\M^+: \Lambda_\delta' \to \Lambda_\delta$, defined as 
\[
\M^+ = (\Bd^+)^T \tAd \Bd^+,
\]
which corresponds to the bilinear form $\m^+ : \Lh'\times \Lh'$ defined as
\[
\m^+(\phi,\psi) = \ta(\Bd^+ \phi,\Bd^+ \psi).
\]

By combining the bilinear form $\m^+$ with the mapping $\Embed$ defined by \eqref{defK}, we finally build  our preconditioner $\hM^{-1}: \hLh' \to \hLh$ as the operator corresponding to the bilinear form  $\hm : \hLh' \times \hLh' \to \mathbb{R}$ defined as
\[
\hm(\tphi,\tpsi) = \m^+(\Embed \tphi,\Embed \tpsi).
\]

We have the following theorem, the proof of which is fundamentally the same as the proof of the analogous result for the FETI method in the domain decomposition framework (see \cite{TW05}). {For the sake of completeness we give the proof in an appendix.}

\begin{thm}\label{thm:preconditioning} Let 	$\Pi_d: \Vh \to \Vh$ denote the $\Vscal$-orthogonal projection operator onto  the $\Vscal$-orthogonal complement 
	$(\ker \Bd)^{\perp}$
	of $\ker \Bd$ in $\Vh$ and, for $v \in \Vh$ and $L:\Vh \to \Vh$ linear operator, let 
	\[\normta{v}^2 = \snormta{v}^2 + \kappa_* \| \bar v \|^2_{V}, \qquad 	\| L \|_{\ta\to\ta} = \sup_{w\in \Vh} \frac{\| L w \|_{\ta}}{\| w \|_{\ta}},\] $\bar v \in Z$ denoting the $V$-orthogonal projection of $v$.
	Then it holds
	\[
	{	\kappa(\hM^{-1} \hSh) \lesssim C(\ta,\Vh,d,\sigma) }
	\]
	with
	\[
	C(\ta,\Vh,d,\sigma) = \frac {K^*}{\kappa_*}\opNorm{\ta}{\ta}{\Pi_d}^2
	\left(1 + \opNorm{\ta}{\ta} {\Bd^+ \Qsigma  \Bd} \right)^2.
	\]
\end{thm}

\begin{rem}
	As, thanks to \eqref{boundatilde}, $\snormta{v} = 0$ if and only if $v \in Z$,  
	$\normta{\cdot}$ is indeed a norm on $\Vh$.
\end{rem}

	%
	%
	%
	%
%

\subsection{Efficiency of the preconditioner}\label{sec:efficiency}
While the operators $\Ah$, $\Bd$  and $\Gd$ are substantially  given by the continuous problem, there is, a priori, a certain freedom in the choice of the space $\Vh$ and of the of the bilinear forms $\sigma^*$ and $d$. These have to be chosen in such a way that the resulting preconditioner $\hM^{-1}$ is cheap (or easily parallelizable), and that, at the same time, the norms of $\Pi_d$ and  $ \Bd^+\QsigmaT\Bd$ stay bounded. How this can be done, and whether this can be done efficiently, depends on the properties of the actual problem considered. Giving a general abstract recipe simultaneously allowing for efficient parallel implementation and optimal or quasi optimal condition number bounds, by solely relying on the  assumptions made up to now is, in our opinion,  not feasible.
We can, nevertheless, guarantee efficiency of our abstract preconditioning strategy, provided the continuous multiplier space $\Lambda$ has some suitable additional localization property. More precisely, we make the following assumption. 
\begin{assmp}\label{localmult}
	There exist linearly independent subspaces $\Lm \subseteq \Lambda$ such that 
	\begin{enumerate}
		\item $\oplus \Lm$ is dense in $\Lambda$;
		\item for all $\ell$ there exist $\kplus$, $\kminus$ such for all $k \not \in \{\kplus,\kminus\}$,  $v \in V_k$ and $\lambda \in \Lm$ imply $b_k(v,\lambda) = 0$.\label{locm2}
	\end{enumerate}
\end{assmp}

Essentially, Assumption \ref{localmult} requires that, at the continuous level, (a dense subspace of) the multiplier space can be decomposed in subspaces each ``seeing'' at most two subproblems. Under such an assumption  it is always possible to choose $\Lh$ as $\Lh = \oplus \Lh^\ell$, $\Lh^\ell \subseteq \Lm$.  {We remark that such an assumption could be relaxed by allowing each $\Lambda^\ell$ to ``see'' at most $N$ subspaces $V^k$, but for the sake of notational simplicity we do not address such an option, which can however be tackled by a similar strategy. }

We let $\boundary{k} = \{ \ell: \ \kplus = k \ \text { or }, \kminus = k\}$ denote the indexes of all the subspaces $\Lambda^\ell$ that ``see'' the space $\Vh^k$, and $ \neighbors{\ell} = \boundary{\kplus} \cup \boundary{\kminus}$ denote the indexes of all the $\Lambda^{\ell'}$ that share with $\Lambda^{\ell}$ a neighboring space $V_h^k$.
For all $k$, we introduce subspaces
$V^k_0 = \ker \Bh^k$ and
\[
V^k_\ell = \{v \in \Vh^k, \ \ta^k(v,w) = 0, \ \forall w\in V^k_0\ \text{ and }\  b^k(v,\mu) = 0, \ \forall \mu \in \Lambda^i,\ \forall i\not= \ell\}.
\]
\NOTE{We have the splitting
	\[
	v = v_0 \oplus v_\ell.
	\] 
	We have, for $w \in V_0$
	\[
	\ta(v,w) = \ta(v_0,w).
	\]
	If $\ta^k$ coercive on $\ker B^k$, this implies
	\[
	\| v^0 \|^2_{V^k} \lesssim \ta(v^0,v^0) = \ta(v,v^0) \lesssim \| \ta \|_{V^k} \| v^0 \|_{V^k}
	\]
	which implies \(
	\| v^0 \|_{V^k} \lesssim \| v \|^2_{V^k}.
	\)
	If we test  $v$ with $w \in V_\ell^k$ we can write
	\[
	\ta(v,w_\ell) = \ta(\sum_\ell v_\ell,w_\ell) = \sum_\ell \ta(w_\ell,v_\ell). 
	\]
	There is no reason to think that $\ta(w_i,v_i) = 0$ if $i \not= \ell$. Then we do not have an a priori upper bound for the splitting only depending on the cardinality of the boundary.
}
We make the following assumptions on the space $\Vh$
\begin{assmp}\label{assmp:Vhsplit} For all $k$ the space $\Vh^k$ can be split as
	\[
	\Vh^k = V^k_0 \oplus_{\ell\in\boundary{k}} V^k_\ell
	\]
	and there exist constants $\cVsplit$ and $\CVsplit$ such that for all $k$, for all $v^k = v^k_0 \oplus_{\ell\in\boundary{k}} v^k_\ell$ it holds
	\[
	\cVsplit 	\normta{v^k}^2 \leq \normta {v^k_0}^2 + \sum_{\ell\in\boundary{k}} \normta {v^k_\ell}^2 \leq \CVsplit  \normta {v^k}^2.
	\]
\end{assmp}

%

\newcommand{\cbat}{\widetilde c}
\newcommand{\Cbat}{\widetilde  C}

%

We remark that, by triangle inequality, $\cVsplit$ is bounded from below by a constant depending only on the cardinality of $\boundary{k}$, while the constant $\CVsplit$ depends on the properties of the norms, bilinear forms and spaces involved. In particular, in the framework of the FETI method for domain decomposition, $\CVsplit$ behaves like $\log(\delta/H)^2$, $\delta$ and $H$ denoting respectively the meshsize and the diameter of the subdomains. 
Under these assumptions we can define the scalar product  $d$ as
\begin{equation}\label{defd}
d(u,v) = \sum_k d^k(u^k,v^k), \qquad d^k(u^k,v^k) = \ta^k(u^k_0,v^k_0) + \sum_{\ell\in \boundary{k}} \ta^k(u^k_\ell,v^k_\ell).
\end{equation}
We claim that, for such a definition of $d$, the pseudo inverse $\Bh^+$ can be evaluated by solving, in parallel, a set of ``local'' problems.
Indeed, we have the following lemma.
\begin{lem}  Let $\phi \in \Lh'$ be given, and for all $\ell$ let $ (v_\ell^+,v_\ell^-) \in V^\kplus_\ell \times V^\kminus_\ell $, $\zeta_\ell \in \Lambda^\ell$ be the solution to 
\begin{gather}
	\ta^\kplus(v_\ell^+,w^+) - b^\kplus(w^+,\zeta_\ell) = 0, \qquad \forall w^+ \in V^\kplus_\ell, \label{local1}\\
	\ta^\kminus(v_\ell^-,w^-) - b^\kminus(w^-,\zeta_\ell) = 0, \qquad \forall v^- \in V^\kminus_\ell,\label{local2}\\
	b^\kplus(v_\ell^+,\mu) + b^\kminus(v^-_\ell,\mu) = \langle \phi,\mu \rangle, \qquad \forall \mu \in \Lm.\label{local3}
\end{gather}
Then, letting $v_\ell \in V_h$ be defined by 
\[
v_\ell^\kplus = v_\ell^+, \qquad v_\ell^\kminus = v_\ell^-, \qquad v_\ell^i = 0,\quad i \not\in \{\kplus,\kminus\}
\]
it holds that
\begin{equation}\label{def_B+efficient}
B^+ \phi = \sum_\ell v^\ell.
\end{equation}
\end{lem}

\begin{proof} We need to prove  that for $B^+$ defined by \eqref{def_B+efficient} and $d$ defined by \eqref{defd}, and for a suitable $\zeta \in \Lambda$ such that \eqref{charBplus} holds. Then, uniqueness of the solution of \eqref{charBplus} implies that \eqref{defd} holds.
We take $\zeta = \sum_\ell\zeta_\ell$  and we have
\begin{multline*}
	d(v,w) = d(\sum_\ell v_\ell,w) = \sum_\ell d(v_\ell,w) = \sum_\ell
	(\ta^\kplus(v^+_\ell,w^\kplus) + \ta^\kminus(v^-_\ell,w^\kminus) ) \\= 
	\sum_\ell (b^{\kplus}(w_\ell^\kplus,\zeta_\ell) + b^{\kminus}(w_\ell^\kminus,\zeta_\ell)) =
	\\=
	\sum_k \sum_{\ell\in\boundary{k}} b^k(w_\ell^k,\zeta_\ell) 
	= \sum_k \sum_{\ell\in\boundary{k}} b^k(w_\ell^k,\zeta) 
	\\= \sum_k  b^k( \sum_{\ell\in\boundary{k}} w_\ell^k,\zeta) + \sum_k b^k (w^k_0,\zeta)=
	\sum_k  b^k( w^k,\zeta)  = b(w,\zeta),
\end{multline*}
where we used \eqref{local1} and \eqref{local2}, and where we could add $b^k(w^k_0,\zeta)$ in the last line as $w^k_0 \in \ker 
B_\delta^k$.
That is we have
\[
d(v,w) - b(w,\zeta) = 0, \qquad \text{ for all } w \in V_h.
\]
Moreover, using \eqref{local3}, we have
\begin{multline}
	b(v,\mu) = b(\sum_\ell v_\ell,\mu) = \sum_\ell b(v_\ell,\mu) =
	\sum_\ell (b^{\kplus}(v^+_\ell,\mu) + b^{\kminus}(v^-_\ell,\mu)) \\=
	\sum_\ell (b^{\kplus}(v^+_\ell,\mu_\ell) + b^{\kminus}(v^-,\mu_\ell)) =
	\sum_\ell\langle \phi, \mu_\ell\rangle =
	\langle \phi, \sum_\ell \mu_\ell\rangle = \langle \phi,\mu \rangle.
\end{multline}
The couple $(v,\zeta)$ then coincides with the unique solution to \eqref{charBplus}, and then our claim that $\Bd^+ \phi = v$ is proven.
\end{proof}
%
%
%
%
%
%
%
%
%

\begin{rem}	This approach to the construction of the pseudo inverse $\Bh^+$ corresponds to what, in the classical domain decomposition framework, is referred to as ``deluxe''  scaling \cite{widlund2016bddc}.
\end{rem}

\begin{rem}\label{rem:glob4loc} We observe that, under some mild assumption, $B^+$ can be evaluated without actually explicitly constructing the subspaces $V^k_\ell$, whose definition can instead be directly embedded in the local problems.
Indeed, let $\widehat v_\ell = (\widehat v_\ell^+,\widehat v_\ell^-) \in \Vh^\kplus \oplus \Vh^\kminus$, $\xi_\ell \in \oplus_{i \in \neighbors{\ell}} \Lambda^{i}$ be the solution to the modified local problem
\begin{gather}
	\ta^\kplus(\widehat v_\ell^+,w^+) - b^\kplus(w^+,\xi_\ell) = 0, \qquad \forall w^+ \in \Vh^\kplus, \label{localb1}\\
	\ta^\kminus(\widehat v_\ell^-,w^-) - b^\kminus(w^-,\xi_\ell) = 0, \qquad \forall v^- \in \Vh^\kminus, \label{localb2}\\
	b^\kplus(\widehat v_\ell^+,\mu) + b^\kminus(\widehat v^-_\ell,\mu) = \langle \phi,\mu_\ell \rangle, \qquad \forall \mu  = \sum_\ell \mu_\ell \in  \mathop\oplus_{i \in \neighbors{\ell}} \Lambda^{i}.\label{localb3}
\end{gather}

\NOTE{ Problem \eqref{localb1}--\eqref{localb3} is written in strong form as 
	\[
\left(	\begin{array}{ccc}
		\tAh^{\kplus} & 0 & - (B^{\kplus})^T\\
0 & 	\tAh^{\kminus} & - (B^{\kminus})^T\\	
B^{\kplus} & B^{\kminus} & 0
\end{array}\right)
\left(
\begin{array}{c}
	v_\ell^+\\
	v_\ell^-\\
	\xi_\ell
\end{array}
\right) = \left(
\begin{array}{c}
	0\\
	0\\
	\phi
\end{array}
\right) 
	\]
	
}

Assuming such a problem is well posed, which requires the validity of a slightly stronger inf-sup condition than the one in Assumption \ref{infsupprecond}, we have that $\widehat v_\ell^+ = v_\ell^+$ and $\widehat v_\ell^- = v_\ell^-$. Indeed, we easily see that $\widehat v_\ell^+ \in V^\kplus_\ell $ and $\widehat v_\ell^- \in V^\kminus_\ell $. In fact $\widehat v^+$ is orthogonal to $V^\kplus_\ell$ for $\ell = 0$ (by \eqref{localb1}) and for $\ell \in \boundary{k_\ell^+}$ (by equation \eqref{localb3}), and analogously for $\widehat v_\ell^-$. Then, if we test \eqref{localb1} (resp. \eqref{localb2}) with $w^+ \in V^\kplus_\ell $ (resp. $w^- \in V^\kminus_\ell $), the contribution of $\xi_\ell \in \Lambda_\ell$, $\ell \not = m$ vanishes, and we get the same equation.
\end{rem}

To assess the efficiency of the preconditioner it remains to estimate the condition number of the preconditioned matrix, which, thanks to Theorem \ref{thm:preconditioning}, reduces to bounding the norms of the projector  $\Pi_d$ and of the operator $\Bh^+ \Qsigma \Bh$.
This is the object of the following lemma.

\begin{lem}\label{lem:boundsigma} Under assumptions \ref{localmult}, it holds that
\begin{equation}\label{boundBQB}
\opNorm{\ta}{\ta}{\Pi_d} \lesssim \sqrt{\frac {C^\flat} {\cVsplit}}, \qquad
\opNorm{\ta}{\ta}  {\Bh^+ \Qsigma \Bh}  \lesssim   
\sqrt{\frac {C^\flat} {\cVsplit}}.
\end{equation}
\end{lem}

\begin{proof} 
In view of the definition of $d$ and of $\Pi_d = 1-\Qd$, by \eqref{dboundedness}  we have, for all $w \in \Vh$,
\[
\normta{\Qd v}^2 \lesssim \snormta{\Qd v}^2 \lesssim \frac 1 {\cVsplit} d(\Qd v,\Qd v) \leq \frac 1 {\cVsplit} d(v,v) \lesssim \frac{\CVsplit}{\cVsplit} \snormta{v}^2,
\]
whence
\[
\normta{\Pi_d v}^2 \lesssim \normta v^2+ \normta{\Qd v}^2 \lesssim \frac{\CVsplit}{\cVsplit}\normta v^2.\] 
Thanks to the coercivity of $\sigma^*$ on $\Range \Gd$, we easily see that $\| \Qsigma \|_{\Lambda'\to \Lambda'} \lesssim 1$,  the implicit constant in the inequality depending on the implicit constant in the coercivity and continuity bounds \eqref{boundsigmastar}. Now we observe that, for $z\in Z$, thanks to the definition \eqref{defnormatilda}, we have that
\[ \| z \|_{\ta} =
\sqrt{\kappa_*} \| z \|_{V} \lesssim \sqrt{\kappa_*} \sup_{\mu \in \Lambda^0} \frac{ b (z,\mu) } {\| \mu \|_{\lambda}} =\sqrt{\kappa_*} \sup_{\mu \in \Lambda^0} \frac{ \langle G z,\mu \rangle } {\| \mu \|_{\Lambda}} \lesssim \sqrt{\kappa_*} \| G z \|_{\Lambda'}.
\]
Then, with $z^* \in Z$ such that $G z^* = \Qsigma \Bh w$ we can write
\begin{multline*}
\| \Bh^+ \Qsigma \Bh w \|^2_{\ta} =  \| \Bh^+ \Bh z^* \|^2_{\ta} =  \| \Pi_d z^* \|^2_{\ta} \lesssim 
{\frac {C^\flat} {\cVsplit}}  \|  G z^* \|^2_{\Lambda_\ell'} \\ \lesssim  { \frac {C^\flat} {\cVsplit}} \kappa_* \| \Qsigma \Bh w \|_{\Lambda'} \lesssim { \frac {C^\flat} {\cVsplit}} \kappa_* \| w \|_V \lesssim 
\frac {C^\flat} {\cVsplit} \normta{w}
\end{multline*}
which gives us \eqref{boundBQB}.
\end{proof}

Combining Theorem \ref{thm:preconditioning} with Lemma \ref{lem:boundsigma} we obtain the following corollary.
\begin{cor}\label{cor:4.10} Under the assumptions of Theorem \ref{thm:preconditioning} and Lemma  \ref{lem:boundsigma} we have that
\begin{equation}\label{finalcondbound}	\kappa(\hM^{-1} \hSh) \lesssim \frac {K^*} {\kappa_*}{\frac {C^\flat} {\cVsplit}}.
\end{equation}
\end{cor}

\begin{rem}
Depending on the choice of $\widetilde a$ the expression at the right hand side of \eqref{finalcondbound} might simplify. If we choose $\widetilde a = a$, we have that $K^* = \kappa_* = 1$. If, instead, we choose $\widetilde a$ with a suitable block diagonal structure, we will have that $\CVsplit = \cVsplit = 1$. 
\end{rem}

\subsection{The solution recipe}
Let us conclude this section a recipe summarizing the building blocks needed in the solution process, and how these can be assembled starting from a number of elementary operators and spaces, some of which depend on the continuous / discrete problems \eqref{schur} and \eqref{feti_discrete}, and some of which have to be chosen case by case. We focus on the framework considered in Section \ref{sec:efficiency}.
he following spaces and operators must be chosen by the user, for the construction of the preconditioner
\begin{enumerate}[(A)]
	\item The bilinear  form $\sigma: \Lambda^0 \times \Lambda^0 \to \mathbb{R}$, satisfying \eqref{boundsigmastar}.
	\item The auxiliary space $\Vh = \prod_k \Vh^k$ satisfying \eqref{infsupprecond}.
	\item The bilinear forms $\ta^k: \Vh^k \times \Vh^k \to \mathbb{R}$ satisfying \eqref{boundatilde}.
\end{enumerate}

\newcommand{\Em}{I^\Lambda_\ell}
\newcommand{\Rm}{R^\lambda_\ell}

The action of the following linear operators on function of the respective domain space must be implemented.
\begin{enumerate}[(a)]
	\item The linear operator $S_h = B A_h^+ B^T: \Lh \to \Lh'$. This will involve a call to the black-box approximate pseudo-solver $A^+_h$.
	\item The linear operator $\Sigma^*: \Lh' \to \Lh$ corresponding to the bilinear form $\sigma^*$. This can be constructed according to Remark \ref{rem:buildsigmastar} starting from $\sigma$.
	\item The natural injection operators $I^V_k: \Vh^k \to \Vh$ and their transpose, the natural restriction operators $R^V_k: \Vh' \to (\Vh^k)'$. 
		\item The linear operators $\tAh^k: \Vh^k \to (\Vh^k)'$
		\item The linear operators $\Bd^k: \Vh^k \to \Lh'$ and their transpose $(\Bd^k)^T: \Lh \to (\Vh^k)'$.
	
	\item The natural injection operators $\Em: \Lh^\ell \to \Lh$ and their transpose, the natural restriction operators $\Rm: \Lh' \to (\Lh^\ell)'$
	\item The linear operator $G: Z \to \Lh'$ and its transpose $G^T : \Lh \to Z'$

	\item The local solvers
	$D_\ell^{-1}: (\Lh^\ell)' \to \Vh^{\kplus} \times \Vh^{\kminus}$ 
	 returning the components $[v_\ell^+, v_\ell^-]^T$ of the solution to the local saddle point problem \eqref{localb1}--\eqref{localb3}.
	\end{enumerate}

\NOTE{
\[
B^+ \phi = \sum_\ell v_\ell = \sum_\ell  = \sum_\ell [I_{\kplus}^V\ I_\kminus^V ] D_\ell^{-1} \phi\]

\begin{multline*}
m^+(\phi,\psi) = \sum_k \ta^k (B^+ \phi, B^+ \psi) = \sum_k \ta^k (\sum_\ell \phi_\ell, \sum_{m'} \psi_{m'}) = 
\sum_\ell \sum_{m'} \sum_k \ta^k(B^+ \phi,B^+ \psi)
\\
\sum_\ell \sum_{m'} \sum_k \ta^k(
\sum_{k\in\mathcal{N}(m)} I^V_k D_\ell^{-1} \phi,\sum_{k'\in\mathcal{N}(m')} I^V_k D_{m'}^{-1} \phi
B^+ \psi
) = \sum_k \sum_{m,m'\in\partial k} \ta^k (I^V_k D_\ell^{-1}\phi,I_k^V D_{m'}^{-1}) \\= 
\sum_k \sum_{m,m' \in \partial(k)} \langle D_{m'}^{-T} R_k^V \tAh^k I_k^V D_\ell^{-1} \phi, \psi \rangle
\end{multline*}
}

Starting from these operators we can evaluate the different other operators involved in the solution process. 
	\newcommand{\kplusp}{k^+_{\ell'}}
\newcommand{\kminusp}{k^-_{\ell'}}

\subsubsection*{The projection operator $\Qsigma$} The projection operator $\Qsigma: \Lh'\to \Lh'$ takes the form
	\[
	\Qsigma =   \Gd (\Gd^T \Sigma^* \Gd)^{-1} \Gd^T \Sigma^* .
	\]
\subsubsection*{The mapping $\Embed$} The mapping $\Embed: \Lh' \to \Lh'$, selecting, for each element $\phi$ of $\Lh'$, the canonical representative of the equivalence class $\tilde \phi = [\phi] \in \hLh$ (see Remark \ref{rem:Embed}), and the projector $\Pi_\sigma: \Lh \to \Lh$ by which the problem is reduced to $\hLh$, have the form 
	\[
	\Embed = (\Id{\Lh'} - \Qsigma), \qquad 	\Pi_\sigma = \Id{\Lh} - (\Sigma^*)^T \Gd (\Gd^T \Sigma^* \Gd)^{-T} \Gd^T.
	\] 
\subsubsection*{The pseudo-inverse $\Bd^+$} The pseudo-inverse $\Bd^+: \Lh'\to \Vh$ is then evaluated as
	\[
	\Bd^+ = \sum_\ell \left(
\begin{array}{cc}
	I^V_{\kplus} & I^V_{\kminus}
\end{array}\right) D^{-1}_\ell R^\Lambda_\ell
	\]
\subsubsection*{The preconditioner} The preconditioner $M^+: \Lh'\to \Lh$ has then the form
	\[
	M^+ = (\Bd^+)^T \tAh \Bd^+ = \sum_\ell  \sum_{\ell' \in \neighbors{\ell}}   I^\Lambda_{\ell'} D^{-T}_{\ell'} \underbrace{\left(
	\begin{array}{c}
		R^V_{\kplusp} \\ R^V_{\kminusp}
	\end{array}\right) \tAh \left(
\begin{array}{cc}
I^V_{\kplus} & I^V_{\kminus}
\end{array}\right) }_{I}D^{-1}_\ell R^\Lambda_\ell
	\]
	Remark that for $k \not=k'$ we have that $R^V_{k'} \tAh I^V_k = 0$, while $R^V_{k} \tAh I^V_k = \tAh^k$, which eventually simplifies $I$ for  $\ell\not=\ell'$.
The operator $M^+$ is then combined with $\Embed$ to form the preconditioner 
$\hM^{-1} = \Embed^T M^+ \Embed$.

\section{Examples}


The above framework has multiple possible applications, and it can be used for the design of coupling approaches in multiphysics, multiscale or mixed-dimensional problems. The coupling of FEM and BEM \cite{BF19}, the Arlequin method (\cite{BD98}), the surface/bulk coupling arising in {proton transport across biological membranes} \cite{GeMeSt02} or in the study of the electrostatic potential in one-layer material semiconductor devices \cite{jourdana2023interface}, the flux in branching fracture networks \cite{PED10, berrone2024stabilized}, can all be cast within our abstract framework, which allows to use previously implemented and optimized solvers of possibly heterogeneous nature for the solution of the individual problems. Of course, the most obvious application is the one to domain decomposition, from which we openly drew inspiration in designing our coupling approach. In the following sections we will show how casting two known domain decomposition approaches in the abstract framework allows first and foremost to design algorithms aimed at coupling black-box subdomain solvers, but also to design new variants of some known preconditioner.


\newcommand{\kk}{\kappa_k}
\newcommand{\kkplus}{\kappa_{k_\ell^+}}
\newcommand{\kkminus}{\kappa_{k_\ell^-}}

\subsection{Domain decomposition with Neumann  coupling}\label{dd}
Let $\Omega \subset \Real^n$, $n=2,3$ be a polygonal/polyhedral domain with boundary $\partial\Omega$.
We consider the following simple model problem:  given $f\in
L^2(\Omega)$,  find $u$ satisfying
\begin{equation}
	\label{eq:poisson}
	- \nabla \cdot (\kappa \grad u) = f, \mbox{~in~}\Omega, \qquad u = 0 \mbox{~on~}\partial\Omega.
\end{equation}
We let $\Omega$ be split as the non-overlapping union of polygonal/polyhedral subdomains: $\overline{\Omega} =\cup_{k=1}^K\overline\Omega^k$ and set $\Gamma^k = \partial \Omega^k$, $\Gamma = \cup \Gamma^k$ denoting the skeleton of the decomposition. We assume that the subdomain $\Omega^k$ are chosen in such a way that 
$\kappa\vert_{\Omega^k} =  \kk$, $\kk \in \mathbb{R}$ positive constants, possibly very different from each other. 
 For each $k$ we let 
\[V_k = \{ u \in H^1(\Omega^k), ~ u|_{\partial\Omega \cap \Gamma^k} = 0\}, \qquad \| u \|_{V^k} = \kk^{1/2} \| u \|_{1,\Omega^k}.\]

Letting $\nu^k$ denote the outer unit normal to the subdomain $\Omega^k$,
on the skeleton $\Gamma$ we next select a normal direction $\nu$, taking care that on $\partial \Omega$ $\nu$ coincides with the outer normal to $\Omega$.  We let $\Lambda$ denote the closure of $L^2(\Gamma)$ with respect to the norm
\[
\| \lambda \|_{-1/2,\Gamma}^2 = \sum_k \kk^{-1} \| (\nu\cdot \nu^k) \lambda \|_{-1/2,\Gamma_k}^2, \qquad \| \lambda^k \|_{-1/2,\Gamma_k} = \sup_{v \in V_k} \frac{\int_{\Gamma_k} \lambda^k v}{\| v \|_{1,\Omega^k}}.
\]
The space $\Lambda$ is well suited to describe the trace on $\Gamma$ of the normal flux $\lambda = \kappa \nabla u \cdot \nu$  of the solution to our problem (we recall that while $\kappa$ and $\nabla u$ may jump across $\Gamma$, $\kappa \nabla u \cdot \nu$ is continuous, so that $\lambda$ is well defined). Equation \eqref{eq:poisson} can be split as a system of coupled local problems as: find $(u^k)_k \in V = \prod_k V_k$, $\lambda \in \Lambda$, such that 
\begin{gather}\label{eq:neumann1}
	\int_{\Omega^k} \kk \nabla u^k \cdot \nabla v^k  - \int_{\Gamma^k} (\nu \cdot \nu^k) \lambda v^k = \int_{\Omega^k} f v^k, \ \text{ for all } v^k \in V_k,\  k \in \{1,\hdots,K\},\\ 
	\sum_k \int_{\Gamma^k} u^k (\nu \cdot \nu^k) \mu = 0,\ \text{ for all } \mu \in \Lambda.\label{eq:neumann2}
\end{gather}
This is the problem from which the abstract framework took inspiration, and consequently it clearly falls within its purview. 

\

\newcommand{\order}{l}

We can now approximate $\Lambda$ with discontinuous piecewise  polynomials of order $\order$ on a shape regular 
mesh $\TG$.  In  the subdomains we consider arbitrary discretizations  yielding approximate solvers $A_{h,k}^+$, which we will treat as black boxes, only assuming that Assumption \ref{consist_weak} holds. These discretizations need neither be of the finite element family, nor they need  be of the same type in all subdomains. We could consider meshless methods or even linear neural network discretizations, such as extreme learning machines, in some subdomains and finite elements in other subdomains.  Lemma \ref{coerc_sh} implies that, provided $h/\delta < C_\Lambda$, Problem \ref{feti_discrete} is well posed and it can be solved by the approach proposed in Section \ref{sec:feti}. Efficiency will be attained by constructing a preconditioner according  the recipe presented therein. 

\begin{rem}
	We could also use continuous polynomials, only allowing discontinuities at the cross points (in two dimensions) or on the wirebasket (in three dimensions), where we also modify the space $\Lh$ in the spirit of the mortar method \cite{bernardi1993domain,belgacem1999mortar}.
\end{rem}

We build the bilinear form $\sigma^*$ according to the strategy suggested in Remark \ref{rem:buildsigmastar}. The space $\Lambda_0$ can be chosen as the space of discontinuous functions assuming constant values on the edges (in 2D) or faces (in 3D) of the subdomains. A bilinear form $\sigma$ satisfying \eqref{condsigma} can be defined as follows:
\begin{equation}\label{eq:Neu_sigma}
\sigma(\lambda,\mu) = \sum_k \kk^{-1} \left(H_ k^{-n} \left(\int_{\partial\Omega^k} (\nu\cdot\nu^k) \lambda\right)  \left(\int_{\partial\Omega^k} (\nu\cdot\nu^k)  \mu \right) + H \int_{\partial\Omega^k} \lambda \mu \right).
\end{equation}
The following lemma, where the implicit constants depend on the shape regularity of the subdomains, but not on their number or size, can be proven by exploiting standard inverse inequalities and \cite[Lemma 2.1]{bertoluzza2021polygonal}.
\begin{lem}
	For all $\lambda \in \Lambda_0$ it holds that
	\[
	\sigma(\lambda,\lambda) \simeq \| \lambda \|_\Lambda^2.
	\]
\end{lem}
\NOTE{
Proof of the lemma. First of all, by Lemma 2.1 of the M2AN paper we have that, since
\[
\| v \|_{1,\Omega^k} \simeq | \Omega^k |^{1/2} |  \fint_{\Omega^k} v | + | v |_{1,\Omega} \simeq H_k^{d/2}  |  \fint_{\Omega^k} v | + | v |_{1,\Omega}
\]
then
\[
\| F \|_{-1,\Omega^k}  = \sup_{v\in H^{1}(\Omega^k)}  \frac{\langle F, v \rangle}{\| v \|_{1,\Omega^k}} \simeq H_k^{d/2} | \langle F , 1 \rangle | + | F |_{-1,\Omega^k}.
\]
We set $F = \gamma_{\partial\Omega^k}^* (\nu\cdot \nu^k) \lambda$ with $\lambda \in \Lambda_0$. We have
\[
H_k^{d/2} | \langle F , 1 \rangle | + | F |_{-1,\Omega^k} = H_k^{-d/2} | \int_{\partial\Omega^k} (\nu\cdot \nu^k) \lambda | + \sup_{{\phi\in H^1(\Omega^k)}\atop{\int_{\Omega^k} \phi = 0}}   \frac{\int_{\partial\Omega^k}(\nu\cdot \nu^k) \lambda \phi}{| \phi |_{1,\Omega^k}}.
\]
Now we can write
\begin{multline*}
\int_{\partial\Omega^k} (\nu\cdot \nu^k) \lambda \phi = \int_{\partial\Omega^k} (\nu\cdot \nu^k) \lambda (\phi - \fint_{\Omega^k} \phi) \leq H^{1/2} \| \lambda \|_{0,\partial\Omega} H^{-1/2} \| \phi - \fint_{\Omega^k} \phi \|_{0,\partial\Omega^k}\\
\lesssim  H^{1/2} \| \lambda \|_{0,\partial\Omega} \left(H^{-1} \| \phi - \fint_{\Omega^k} \phi \|_{0,\Omega^k} + |\phi |_{1,\Omega^k}
 \right)
 \lesssim H^{1/2} \| \lambda \|_{0,\partial\Omega} |\phi |_{1,\Omega^k},
\end{multline*}
From which we obtain the upper bound. We obtain the lower bound by a scaling argument. 
}

 The first step in constructing the preconditioner  is choosing the auxiliary spaces  $V_\delta^k$ and the bilinear forms $\widetilde a^k$. The most natural choice is to introduce the mesh $\El$ on $\partial \Omega^k$  induced by the mesh $\TG$, and a mesh $\Thl$ on $\Omega^k$, whose trace on $\Gl$ coincides with $\El$. We then choose $V_\delta^k$ as the set of order $\order +n$ finite elements on $\Thl$, and we can set $\widetilde a^k = a^k$. It is not difficult to realize that this choice leads to what is essentially a FETI preconditioner with ``deluxe'' scaling, the main difference being the fact that, as in a black box framework, the spaces underlying the local black box solvers $A_{h,k}^+$ might not be available, rather than relying on such spaces to construct the preconditioners, we rely on the spaces $V_\delta^k$. Remark that the evaluation of the scalar product $d^k$ defined by \eqref{defd} requires solving local problems in the space $V_0^k$ of ``interior'' functions, for the construction of the spaces $V^k_\ell $. Remark also that the mesh $\Thl$ need not be quasi uniform and it can be taken much coarser in the interior of $\O^k$, thus making the preconditioner cheaper.

\newcommand{\Ehl}{\mathcal{E}_\delta^k}

We can also consider an alternative, inspired by a combination of the theories of Virtual Elements and of substructuring preconditioners \emph  {\`a la} Bramble Pasciak Schatz.  Depending on whether we are in two or three dimensions, we let $V_\delta^k$ be defined as follows:
\begin{enumerate}[\alph{enumi})]
	\item for $d = 2$ we set 
	\begin{equation}\label{eq:2DV_Neu}
	V_{\delta}^k = \{ v \in H^1(\O):  v|_e \in \Poly{\order+2}(e)\, \forall e \in \El, \Delta v = 0 \ \text{in }\Omega^k  \};
	\end{equation}
	\item for $d = 3$, letting $\Ehl$ denote the set of edges of the boundary mesh $\TG$, we set 
	\begin{multline}\label{eq:3DV_Neu}
	V_\delta^k = \{
	v \in H^1(\O): v|_e \in \Poly{1}(e)\ \forall e \in \Ehl, \\ v|_f \in \Poly{\order+3}(f) \ \forall f \in \El, \ \Delta v = 0  \ \text{in }\Omega^k
	\}.
	\end{multline}
\end{enumerate}

We now let $\widetilde a^k$ be defined as
\begin{equation}\label{eq:Neu_tildea}
\widetilde a^k (v,w) = \kk \sum_{\ell \in \partial(k)} \left( \widetilde a^k_\ell (v,w) +  \left(\fint_{\Gamma^k_\ell } v - \fint_{\Gamma^k} v \right) 
\left(\fint_{\Gamma^k_\ell } w - \fint_{\Gamma^k} w \right)\right),
\end{equation}
where $\Gamma^k_\ell $, $\ell \in \partial (k)$ are the edges (in 2D) or faces (in 3D) of the subdomain $\O^k$, and where
\begin{equation}\label{conditionBPS}
\widetilde a^k_\ell (v,v) \simeq | v |^2_{1/2,\Gamma^k_\ell }.
\end{equation}
{ The construction of bilinear forms $\widetilde a^k_\ell $ satisfying \eqref{conditionBPS} is out of the scope of this paper and we refer to \cite{chernov2019harmonic, bramble1986construction} for different options stemming from the theory of virtual elements and of substructuring preconditioners.}
We have the following Lemma.
\begin{lem}\label{lem:5.2} For all $v \in V_\delta^k$ it holds that
	\[  \log (H/\delta)^{-2}\kk | v |_{1,\Omega^k}^2 \lesssim \widetilde a^k(v,v) \lesssim \kk | v |_{1,\Omega^k}^2.
\]
	\end{lem}
This has been proven in two dimensions in \cite{BMP-Neumann}, but the proof there holds unchanged also in three dimensions. In order to define the pseudo inverse $B_\delta^+$ we need to identify the subspaces $V^k_\ell $, $\ell \in \partial(k)$. To this aim
	let $\widetilde V^k_\ell  = \{ v \in V^k_\delta: v = 0 \ \text { on } \Gamma^k \setminus \Gamma^k_\ell \}$ and let $\pi_\ell : V^k \to \widetilde V^k_\ell $ be defined as the first component of the couple $(\pi_\ell \phi,\lambda) \in \widetilde V^k_\ell  \times \Lambda^m$ such that for all $(\psi,\mu) \in \widetilde V^k_\ell  \times \Lambda^m$
\[
\kk	\widetilde a^k_\ell (\pi_\ell  v,\psi) - \int_{\Gamma^k_\ell } \lambda \psi =0, \qquad \int_{\Gamma^k_\ell } \pi_\ell  \phi \mu = \int_{\Gamma^k_\ell } \phi \mu. 
\]
As we have that 
\[
\inf_{\lambda \in \Lambda^m} \sup_{\phi \in \widetilde V^k_\ell } \frac{ \int_{\Gamma^k_\ell } \lambda \phi }{\| \lambda \|_{-1/2,\Gamma^k_\ell } \| \phi \|_{H^{1/2}_{00}(\Gamma^k_\ell )}} \gtrsim 1,
\]
$\pi_\ell $ is well defined. The following lemma holds, where we let $\Im(X)$ denote the image of the operator $X$. 
\begin{lem}
	It holds that $V^k_\ell  = \Im(\pi_\ell )$.
	\end{lem}
\begin{proof} Let $\phi = \pi_\ell  \phi \in \Im(\pi_\ell )$, and let $\psi \in V^k_0 = \ker \Bd^k$ and $\mu \in \Lambda^\ell$, $\ell \not = m$. We have 
	\[
	\ta(\phi,\psi) =\kappa_k^{-1}  \int_{\G_\ell ^k}  \lambda(\phi) \psi = 0, \qquad b(\phi,\psi) = b(\pi_\ell  \phi,\psi) = 0, 
	\]
	which implies that $\Im(\pi_\ell )\subseteq V^k_\ell $. Let us now prove the converse inequality. Let $\phi \in V_\ell ^k$ and assume that $\pi_\ell  \phi \not=\phi$. Then $\psi = \phi - \pi_\ell  \phi \in V_\ell ^k$. We have
\[
\ta(\psi,\xi) = 0, \quad \forall \xi\in V^k_0, \qquad \text{ and } b(\psi,\mu) = 0, \quad \forall \mu \in \Lh^\ell, \ \ell \not = m.
\]
Together with $b(\psi,\mu) = 0$ for all $\mu \in \Lh^m$ this implies $\psi = 0$, and therefore $\phi = \pi_\ell  \phi$, which gives a contradiction.
	
%
	\end{proof}

We can then define the pseudo inverse $\Bh^+$ as the solution to \eqref{local1}--\eqref{local3}, or, by the reasonong in Remark \ref{rem:glob4loc}, by solving for each $\ell$ the local coupled system \eqref{localb1}--\eqref{localb3} that takes here the form: find $v_\ell ^+ \in \widetilde V^{\kplus}_\ell $, $v_\ell ^- \in \widetilde V^{\kminus}_\ell $ such that 
 \begin{gather}
\kkplus	\ta^\kplus_\ell (\widehat v_\ell ^+,w^+) - \int_{\Gamma^\kplus_\ell } w^+ \xi_\ell  = 0, \qquad \forall w^+ \in \widetilde V^\kplus_\ell , \label{localc1}\\
\kkminus	\ta^\kminus_\ell (\widehat v_\ell ^-,w^-) + \int_{\Gamma^\kminus_\ell } w^- \xi_\ell   = 0, \qquad \forall v^- \in\widetilde V^\kminus_\ell , \label{localc2}\\
\int_{\Gamma_\ell } v^+ \mu -
\int_{\Gamma_\ell } v^- \mu
= \langle \phi,\mu \rangle, \qquad \forall \mu  \in  \Lambda_\ell .\label{localc3}
\end{gather}
It is not difficult to ascertain that the solution of \eqref{localc1}--\eqref{localc3} coincides with the solution to \eqref{local1}--\eqref{local3}. We can then solve the former system instead of the latter one, so than we do not need to explicitly compute the spaces $V^\kplus_\ell $ and $V^\kminus_\ell $. 

Thanks to the way we defined $\ta$ we immediately see that Assumption \ref{assmp:Vhsplit} holds with
\[
\cVsplit \simeq \CVsplit \simeq 1.
\] 
The scalar product $d$ is then defined by \eqref{defd}. 
Moreover, by Lemma \ref{lem:5.2} we have that \eqref{boundatilde} holds with 
\[\kappa_* \simeq \frac 1 {|\log(H/\delta)|^2}, \qquad K^* \simeq 1.\] 
Following the steps of section \ref{sec:constructM},
it follows that the choices above with $\sigma$ defined in \eqref{eq:Neu_sigma}, $\widetilde a^k$ defined by \eqref{eq:Neu_tildea} and  $V^k_\delta$ defined by \eqref{eq:2DV_Neu}, or \eqref{eq:3DV_Neu}, leads to a preconditoner $\widehat M$ for which Corollary \ref{cor:4.10} gives the bound
\[	\kappa(\hM^{-1} \hSh) \lesssim |\log(H/\delta)|^2.
\]

We point out that this preconditioner is, to our knowledge, new. It can be regarded as a dual version of the substructuring preconditioner of \cite{bramble1986construction}.

\NOTE{
Boundedness of $\pi_\ell $. We have
\[
| \pi_\ell  v |_{1,\Omega^k} \lesssim \| \pi_\ell  v \|_{H^{1/2}_{00}(\Gamma^k_\ell )} \lesssim  h^{-\varepsilon} | \pi_\ell  v |_{H^{1/2-\varepsilon}_0(\Gamma^k_\ell )} 
\lesssim  | v |_{H^{1/2-\varepsilon}_0(\Gamma^k_\ell )} 
\lesssim \frac{h^{-\varepsilon}}{\varepsilon} | v |_{1/2,\Gamma^k_\ell }
\]
whence
\[
| \pi_\ell  v |_{1,\Omega^k} \lesssim \log(H/\delta) | v |_{1/2,\Gamma^k_\ell }.
\]
Then
\[
\| v \|^2_{1,\Omega^k} \lesssim \| v - \sum_\ell  \pi_\ell  v \|^2_{1,\Omega^2} + \sum_{m} \| \pi_\ell  v \|^2 \lesssim \log(H/\delta)^2 \| v \|^2_{1,\Omega^k}\]
whence
\[
c_\flat = 1, \qquad C^\flat = \log(H/\delta)^2
\]
}

\subsection{Domain decomposition with Dirichlet coupling}
A second possibility for splitting Problem \ref{eq:poisson} into coupled local problems is to introduce, as an independent variable acting as the multiplier, the trace  of the solution on the skeleton $\Gamma$ of the decomposition. The coupled problem then reads: for all $k \in \{1,\hdots,K\}$ find $(u^k,\zeta^k)_k \in H^1(\O^k) \times H^{-1/2}(\G^k)$ and $\lambda \in H^1_0(\O)|_{\Sigma}$ such that for all  $(v^k,\psi^k) \in H^1(\O^k) \times H^{-1/2}(\G^k)$
\begin{gather}\label{eq:diri1}
	\Big(	\int_{\Omega^k} \kappa \nabla u^k \cdot \nabla v^k -  \int_{\Gamma^k} \zeta^k v^k  +  \int_{\Gamma^k} u^k \psi^k \Big)  -  \int_{\Gamma^k} \lambda \psi^k = \int_{\Omega^k} f v^k,\end{gather}
and for all $\mu \in \Lambda$
\begin{gather}
	\sum_k \int_{\Gamma^k} \zeta^k \mu = 0.\label{eq:diri2}
\end{gather}
Such a problem does not directly fall in our abstract framework, as the local problems \eqref{eq:diri1} are not coercive. However, in view of Remark \ref{rem:noncoercive}, we can apply our analysis and solution strategy  to  a suitably {modified, coercive,} version of such a problem. As the average $\fint_{\G^k} \zeta^k$ can be computed directly from the data thanks to the identity
\[
\int_{\G^k} \zeta^k =  \int_{\G^k} \kappa \nabla u^k \cdot \nu^k =  
\int_{\O^k} \nabla \cdot \kappa \nabla u^k = -  \int_{\O^k} f,
\]
we can substitute the unknown $\zeta^k$ for its average free component $\phi^k = \zeta^k - \fint_{\G^k} \zeta^k$.  
We then introduce the spaces
\[
H^1_\text{av}(\Omega^k) = \{ v \in H^1(\Omega^k),\ \fint_{\Gamma^k} v =0\}, \qquad H^{-1/2}_\text{av}(\Gamma^k) = \{ \phi \in H^{-1/2}(\Gamma^k), \ \langle  \phi,1 \rangle = 0\},
\]
respectively endowed with the norms
\[
\| v \|_{H^1_\text{av}(\Omega^k) } = \sqrt{\kk} | v |_{1,\Omega^k}, \quad \| \phi \|_{H^{-1/2}_\text{av}(\Gamma^k)} = \frac 1 {\sqrt{\kk}} | \phi |_{-1/2,\Gamma^k} = \sup_{v \in H^{1}_{av}(\Omega^k)} \frac{\int_{\Gamma^k} \phi v }{| v |_{1,\Omega^k}}.
\]

Then we can rewrite our problem as: look for $(u^k,\phi^k) \in V_k = H^1(\Omega^k) \times H^{-1/2}_\text{av}(\Gamma^k)$ and  $\lambda \in \Lambda = H^1(\Omega)|_{\Sigma}$, solution of a problem of the form: for all $(v^k,\psi^k) \in V_k$, $\mu \in \Lambda$
\begin{eqnarray}\label{eq:coercive_local} \label{stabhybdd1}
a(u^k,\phi^k;v^k,\psi^k) - b(\psi^k;\lambda) &=& \langle F;v^k,\psi^k\rangle\\
\sum_k b(\phi^k;\mu) &=& \langle G^k,\mu \rangle\label{eq:glob_const_coercive}
\end{eqnarray}
with
\begin{multline}
a^k (u^k,\phi^k;v^k,\psi^k) = 	\int_{\Omega^k} \kappa \nabla u^k \cdot \nabla v^k -  \int_{\Gamma^k} \phi^k v^k  \\
+ \int_{\Gamma^k} u^k \psi^k +  [ \D^k u^k - \gamma_{\Gamma^k}^* \phi^k,   \D^k v^k - \gamma_{\Gamma^k}^* \psi^k]_{-1,\O^k},
\end{multline}
and with  $\bar \kappa^k = \fint_{\O^k} \kappa$,
\begin{gather*}
	b^k(\phi^k;\mu) = \bar \kappa^k \int_{\G^k} \phi^k\mu\\
	\langle F;v^k,\psi^k \rangle = \int_{\O^k} f v^k + [f, \D^k v^k - \gamma_{\G^k}^* \psi^k]_{-1,\O^k} + \int_{\Gamma^k} v^k \fint_{\Gamma^k} \zeta^k  \\
	\langle G,\mu \rangle = - \sum_k \int_{\G^k} \mu \fint_{\G^k} \zeta^k.
\end{gather*}
%

In the above expression, $\D^k: H^1(\O^k) \to H^1(\O^k)'$ is the linear operator corresponding to  the semi-scalar product induced by $a^k$,
\[
a^k(u,v) = \langle \D^k u, v \rangle,
\]
 $\gamma_{\Gamma^k}: H^1(\O^k) \to H^{1/2}(\partial\O^k)$ is the trace operator, $\gamma_{\Gamma^k}^*$ denoting its adjoint, and $[\cdot,\cdot]_{-1,\O^k}$ stands for the $H^1_\text{av}(\O^k)'$ scalar product. We recall that the bilinear form $a^k$ does not need to play any role in the design of the black box local solvers (see Remark \ref{rem:noncoercive}). It will however play a role in the design of the preconditioner. To see that the formulation \eqref{eq:coercive_local}--\eqref{eq:glob_const_coercive} is indeed consistent with \eqref{eq:diri1}--\eqref{eq:diri2} consider the Riesz map $\Psi:(H^{1}(\O^k))' \to H^{1}(\O^k)$ and define
 \[
 [ v,   w]_{-1,\O^k} = \langle v, \Psi w \rangle, \quad \mbox{ for } v.w \in (H^1(\O^k))'.
 \]
It then follows that if we set $w^k = \D^k v^k - \gamma_{\Gamma^k}^* \psi^k$ and $\tilde w^k = \Psi w^k$,
\begin{equation*}
	[ \D^k u^k - \gamma_{\Gamma^k}^* \phi^k,  w^k]_{-1,\O^k} = \left< \D^k u^k - \gamma_{\Gamma^k}^* \phi^k, \tilde w^k \right>
	= a^k(u^k,\tilde w^k)  - \left<\phi^k,\tilde w \right>.
\end{equation*}
Using the definition of $\phi^k$ we have
\begin{multline*}
a^k(u^k,\tilde w^k)  - \left<\phi^k,\tilde w \right>
= a^k(u^k,\tilde w^k)  - \left<\zeta^k,\tilde w^k \right>
+ \int_{\Gamma^k} \tilde w^k \fint_{\Gamma^k} \zeta^k\\
 = \int_{\O^k} f \tilde w^k + \int_{\Gamma^k} \tilde w^k \fint_{\Gamma^k} \zeta^k =  [f, \D^k v^k - \gamma_{\G^k}^* \psi^k]_{-1,\O^k} + \int_{\Gamma^k} v^k \fint_{\Gamma^k} \zeta^k.
\end{multline*}
It is not difficult to check that the modified system has a trivial kernel.  Indeed, for $F = 0$, testing \eqref{stabhybdd1} with $(v^k,\psi^k) = (u^k,\phi^k)$ we obtain that
\[
\int_{\Omega^k} | \nabla u^k |^2 + \| \D u^k - \gamma^*_{\Gamma^k} \lambda^k \|^2_{-1,\Omega^k} = 0
\]
which implies that $u^k$ is a constant (and, consequently, that $\D u^k = 0$) and $\gamma_{\Gamma^k}^* \lambda = 0$. This last equation implies that $\lambda = 0$. Then, testing \eqref{stabhybdd1} with $(v^k,\psi^k) = (0,1)$ we obtain
\[
\int_{\Gamma^k} u^k = 0,
\] 
which finally implies $u^k = 0$.

We now consider a discretization of $\Lambda = H^1_0(\O)|_{\Sigma}$ by conforming degree $\order$ finite elements. As far as the local solvers are concerned, we point out that we essentially require them to provide us with approximate  Dirichlet to Neumann mappings. Assumption \ref{consist_weak} reduces to asking that such Dirichlet to Neumann mappings are sufficiently precise. Let us then focus on constructing the preconditioner. 
As in our formulation the control of $\lambda$ is obtained through the test function $\psi$, we can take $V_\delta^k$ of the form $V_\delta^k = \{0\} \times \Phi^k_\delta$, with $\Phi^k_\delta$ such that
\[ \inf_{\lambda \in \Lambda} \sup_{\phi \in \Phi^k_\delta} \frac{\int_{\G^k} \phi \lambda}{| \lambda |_{1/2,\G^k} \| \phi \|_{-1/2,\G^k} } \gtrsim 1, \qquad\text{with}\quad \| \phi \|_{-1/2,\Gamma} = \sup_{w\in H^1(\O^k)} \frac{\int_{\G^k} \phi w }{\| w \|_{1,\Omega^k}}.\]
We know that, provided the mesh is quasi uniform, such inf sup conditions are satisfied for $\Phi_\delta^k = \Lambda|_{\G^k}$.
We can then take any $\Phi^k_\delta$ with $\Phi_\delta^k \supseteq \Lambda|_{\G^k}$. It will be convenient to take $\Phi^k_\delta$ as the space of degree $\order$ finite elements, discontinuous at the subdomain vertices (in 2D) or at the wirebasket (in 3D), and continuous everywhere else. As $\ker A^k$ is trivial, we have that $\hLh = \Lh$ and we can set $\Pi_\sigma = \Id{\Lh}$. We can define the bilinear form $\widetilde a$ as
\[
\widetilde a^k(\phi,\psi) = \sum_{F\text{ face of }\Omega^k} \left( a^k_F(\phi,\psi) + H^{(d-1)}
\left(
\fint_F \phi 
\right)\left(
\fint_F \psi 
\right)
 \right)
 \]
 with, for all face $F$ of the subdomain $\Omega^k$, 
 \[a^k_\verde{F}(\phi,\phi) \simeq \sup_{{\zeta \in H^{1/2}(F)} \atop {\int_F \zeta= 0}} \frac{\int_F \phi\zeta}{| \zeta |_{1/2,F}}
\]
The following Lemma can be obtained by a dual argument to the one underlying the proof of Lemma \ref{lem:5.2}.
\begin{lem}\label{lem:minusprec} For all $\phi \in \Phi^k_\delta$ it holds that
	\[a(0,\phi;0,\phi) \lesssim {\widetilde a(\phi,\phi)} \lesssim \log(H/\delta)^2 a(0,\phi;0,\phi).
	\]
\end{lem}

\NOTE{Proof of the Lemma with the duality argument. I have
	\[  \log (\delta / H)^{-2}| v_\delta |_{1/2,\partial K}^2 \lesssim \underbrace{\sum_{E \text{ macro edge of }K} \left(| v_\delta |^2_{1/2,E} +  | \fint_{E} v_\delta | ^2\right) }_{| v |_{1/2,*}} \lesssim | v_k |_{1/2,\partial K}^2
	\]
	(proven in the Neumann paper). As we deal with the average free flux, seminorms are ok).

	The bilinear form 
	\[
	a(0,\phi;0,\phi) \simeq | \gamma_{\Gamma^k}^* \phi |_{-1,\Omega^k} = \sup_{v \in H^1(\Omega^k)} 
	\frac{\int_{\Gamma^k} \phi v } {| v |_{1,\Omega^k }} = 
	\]

	\begin{gather*}
		\sup_{v \in H^1(\Omega^k)} 
		\frac{\int_{\Gamma^k} \phi v } {| v |_{1,\Omega^k }} \lesssim
		\sup_{v \in H^1(\Omega^k)} 
		\frac{\int_{\Gamma^k} \phi v } {| v |_{1/2,\Gamma^k }}  
		\lesssim \sup_{v \in H^1(\Omega^k)} 
		\frac{\int_{\Gamma^k} \phi v } {| v |_{1/2,*}}
	\end{gather*}
	We have $\Lambda_\delta \subseteq \prod_F H^{-1/2}_\text{av}(F) \times \Poly{0}(F)$ endowed with the product norm. Then the dual is the product of the duals, meaning that
	\[
	\sup_{v \in H^1(\Omega^k)} 
	\frac{\int_{\Gamma^k} \phi v } {| v |_{1/2,*}} \simeq \left(
	\sum_F | \phi |_{-1/2,F}^2 + H^{d-1} | \fint_F \phi |^2 \right)^{1/2} \lesssim \sup_{v \in H^1(\Omega^k)} 
	\frac{\int_{\Gamma^k} \phi v } {| v |_{1/2,*}}.
	\] 
	Now we have that 
	\begin{multline*}
		\sup_{v \in H^1(\Omega^k)} 
		\frac{\int_{\Gamma^k} \phi v } {| v |_{1/2,*}} \lesssim \log(\delta/H)^2 \sup_{v \in H^1(\Omega^k)} 
		\frac{\int_{\Gamma^k} \phi v } {| v |_{1/2,\Gamma^k}} = 
		\log(\delta/H)^2 \sup_{v \in H^{1}(\Omega^k)} 
		\frac{\int_{\Gamma^k} \phi v^H } {| v^H |_{1,\Omega^k}} \\ \lesssim   \log(\delta/H)^2 \sup_{\atop{v \in H^{1}(\Omega^k)}
			v\text{ harmonic}
		} 
		\frac{\int_{\Gamma^k} \phi v } {| v |_{1,\Omega^k}}  \lesssim  \log(\delta/H)^2 \sup_{v \in H^1} \frac{\int_{\Gamma^k} \phi v } {| v |_{1,\Omega^k}}
	\end{multline*}
}

In this case, Assumption \ref{assmp:Vhsplit} is not satisfied, so we will have to rely on Theorem \ref{thm:preconditioning} for estimating the condition number. The bilinear form $\ta$ is, by construction, split as a coarse component plus a block diagonal component, so that it can be inverted efficiently in parallel. We then let $d = \ta$, which automatically yields $\cVsplit \simeq \CVsplit \simeq 1$, while Lemma  \ref{lem:minusprec} yields $\kappa_* = 1$ and $K^* = \log(\delta/H)^2$. Using the choices of $\widetilde a$, $d$ and $V_\delta^k$ defined above we can construct $\widehat M^{-1}$ following section \ref{sec:constructM}. The resulting preconditioner, which is, to our knowledge, novel, is similar to the substructuring preconditioner by Bramble, Pasciak and Schatz, but with a different treatment of the wirebasket. Theorem \ref{thm:preconditioning} yields the bound
\[
\kappa(\widehat M^{-1}S_h) \lesssim |\log(H/\delta)|^2
\]

\begin{rem}
	In a finite element context the former domain decomposition formulation \eqref{eq:neumann1}--\eqref{eq:neumann2} was first discussed in \cite{RT77} and typically results in mortar methods \cite{BM97,belgacem1999mortar, Wohl00}, or alternatively FETI methods, see \cite{TW05} and references therein.
	Remark that, in the framework of the latter domain decomposition formulation \eqref{eq:diri1}--\eqref{eq:diri2},  Problem \ref{feti_discrete}  substantially coincides with the three fields domain decomposition method \cite{BM94} as well as the natural domain decomposition method with non matching grids proposed in \cite{steinbach2005natural}. A coupling using the Dirichlet trace as hybrid variable can also be introduced for mixed finite element method, \cite{ACWY00}. On the discrete level any method for the weak imposition of Dirichlet conditions may be used to realize \eqref{eq:diri1}--\eqref{eq:diri2}. In particular, one may eliminate the local multiplier by applying Nitsche's method in the spirit of \cite{Sten95}, resulting in the hybridised method proposed in \cite{Eg09}. 
	Note however, that if the multiplier is eliminated, in general the coupling scheme can not be used in a black box fashion, since information on both the Dirichlet and Neumann traces are required. 
	Both these methods have been applied in the context of multi scale methods, see \cite{HPV13} and \cite{PS08}, with analysis of a preconditioner in the second reference.
	
\end{rem}

\section*{Acknowledgement}
EB was partially funded by the EPSRC project EP/W007460/1, {\emph{Integrated Simulation at the Exascale: coupling, synthesis and performance}}.  SB is member of the \emph{Gruppo Nazionale di Calcolo Scientifico -- Istituto Nazionale di Alta Matematica (GNCS--INdAM)} and was partially funded by  \emph{D34Health - Digital Driven Diagnostics, prognostics and therapeutics for sustainable Health care} initiative (PNC0000001) and by MUR PRIN 2022 PNRR, grant P2022BH5CB (funded by NextGeneration EU).

\bibliographystyle{amsplain}
\bibliography{coupling}

\newpage

\appendix

\newcommand{\Ad}{\Ah}

\section{Proof of Theorem
	\ref{thm:preconditioning}}
\label{appendix:proof}

In order to prove Theorem \ref{thm:preconditioning} it is sufficient to prove that, for all $\hlambda \in \hLh$ it holds that
\[\frac{1}{ c_*}
\sqrt{s_h(\hlambda,\hlambda)} 
\lesssim \sup_{\tphi \in \hLh'} \frac {\langle \tphi, \hlambda \rangle}
{\sqrt{\hm(\tphi,\tphi)}} \lesssim C^*
 \sqrt{s_h(\hlambda,\hlambda)},
\]
with 
\[
c_*= \sqrt{K^*} \opNorm{\ta}{\ta}{\Pi_d} \left(1 + \opNorm{\ta}{\ta} {\Bd^+ \Qsigma \Bd} \right), \qquad C^* =  \frac 1 {\sqrt{\kappa_*}}.
\]

	%
	%
	%
	%
%
%
%
We start by observing that
\begin{equation}\label{normequivtaV}
	\kappa_*  \| v \|_{V}^2 \lesssim \normta{v}^2 \lesssim K^* \| v \|_V^2.
\end{equation} 
Moreover, we have that
\[
\normta v \lesssim \snormta v 
\qquad \text{ for all $v \in \ker \Bh$}.
\]

\NOTE{Using: boundedness of $V$-orthogonal projection and coercivity of $A_\delta$ on $\ker B_\delta$, which gives us a sort of Poincar\'e bound we have
	\[
	\| v \|_{\ta} = | v |_{\ta} + \kappa_* \| \bar v \|_V \leq | v |_{\ta} + \kappa_* \| v \|_{V} \lesssim  | v |_{V} 
	\lesssim | v |_{\ta} + \kappa_* | v |_{V} \leq 2 | v |_{\ta}
	\]
	
}

We start by observing that, letting $\tphi \in \hLh'$, 
we can split $\Bd^+ \Embed \tphi \in \Vh$ ($\Embed$ defined by \eqref{defK}) as $\Bd^+ \Embed \tphi = z_\tphi + w^\perp_\tphi$, with $z_\tphi \in Z$, and $w_\tphi^\perp \in \Vh$ $V$-orthogonal to $Z$.  As $z \in \ker \opA = \ker \opA^T$, we have
\begin{equation}\label{reducea}
a(\Bd^+ \Embed\tphi,\Bd^+ \Embed\tphi) = a(w^\perp_\tphi,w^\perp_\tphi).
\end{equation}
As $\Bd$ is surjective, which is a consequence of the inf-sup condition  \eqref{infsupprecond}, for some $w' \in \Vh$, it holds that $\Embed \tphi = \Bd w'$. Then we can write
\begin{equation}\label{eqB1}
\Embed \tphi = \Bd w' = \Bd \Bd^+ \Bd w' = \Bd \Bd^+ \Embed \tphi.
\end{equation}

%

Let now $\hlambda \in \hLh$. Using  Lemma \ref{coerc_sh}, \eqref{normequivtaV} and the continuity of $b$, 	equations \eqref{reducea}  and \eqref{eqB1}, we can write
\begin{multline*}
\sqrt{s_h(\hlambda,\hlambda)} \gtrsim \| \hlambda \|_{\Lambda} \gtrsim  \sup_{w \in \Vh \cap Z^\perp} \frac{b(w,\hlambda)}{\| w \|_{V}}  \gtrsim
\sqrt{\kappa_*}
\sup_{w \in \Vh \cap Z^\perp}  \frac{b(w,\hlambda)}{\snormta{w}}
\geq \sqrt{\kappa_*} \sup_{\tphi \in \hLh'}\frac{b(w^\perp_\tphi,\hlambda)}{\snormta{w^\perp_\tphi}
} \\ =\sqrt{\kappa_*} \sup_{\tphi \in \hLh'}\frac{b(\Bd^+\Embed\tphi - z_\tphi,\hlambda)}{
\snormta
		{\Bd^+ \Embed \tphi}}
 =  \sqrt{\kappa_*}
\sup_{\tphi \in \hLh'}\frac{b(\Bd^+\Embed\tphi,\hlambda)}{
	\snormta
		{\Bd^+ \Embed \tphi}	} 
\\=
\sup_{\tphi \in \hLh'}\frac{ \langle \tphi,\hlambda\rangle }{
\snormta
{\Bd^+\Embed  \tphi}
	} = \sqrt{\kappa_*}	\sup_{\tphi \in \hLh'}\frac{ \langle \tphi,\hlambda\rangle }{\sqrt{\hm( \tphi,   \tphi) }},
\end{multline*}
 and we have the upper bound. 

\

Let us now prove the lower bound.  Recalling that $\Bd^+ \Bd = \Pi_d$ (see \eqref{propBdplus}), we have
\begin{multline}
\sqrt{s_h (\hlambda,\hlambda)} \lesssim \| \hlambda \|_\Lambda \lesssim \sup_{w \in \Vh} \frac{b(w,\hlambda)}{\| w \|_V} \leq \sqrt{K^*} \sup_{w \in \Vh} \frac{b(w,\hlambda)}
{
	\normta
	{\Pi_d w}
	}
\frac {
	\normta
{\Pi_d w}
 } 
{\normta w} \\
\leq \sqrt{K^*}
\opNorm{\ta}{\ta}{\Pi_d}  \sup_{w \in \Vh} \frac{b(w,\hlambda)}
{
	\normta
	{\Pi_d w}
	} =  \sqrt{K^*}	\opNorm{\ta}{\ta}{\Pi_d} \sup_{w \in \Vh} \frac{b(w,\hlambda)}
{\normta
		{\Bd^+ \Bd w}
	},
\end{multline}
where we used \eqref{infsupprecond} and \eqref{normequivtaV}.
We also have
\[
b(w,\hlambda) = \langle \Bd w,\hlambda \rangle = 
\langle \Bd \Bd^+ \Bd w, \hlambda \rangle = b(\Bd^+ \Bd w, \hlambda),
\]
whence
\begin{equation}\label{eq:A4}
\sqrt{s_h(\hlambda,\hlambda)} \lesssim \sqrt{K^*}
\opNorm{\ta}{\ta} {\Pi_d }  \sup_{w \in \Vh} \frac{b(\Bd^+ \Bd w ,\hlambda )} { \normta
{\Bd^+	\Bd  w} }.
		\end{equation}
		Now, for $z \in Z$ arbitrary, 
		by the definition of $\hLh$ we have that
		\[
		b(\Bd^+ \Bd z, \hlambda) = b(z ,\hlambda) = 0.
		\]	
	Then	we can write
		\begin{equation}\label{eq:A3}
\sup_{w \in \Vh} \frac{b(\Bd^+ \Bd w ,\hlambda )} 
{ 
\normta	{\Bd^+
		\Bd  w}} 
= \sup_{w\in \Vh} \frac{b(\Bd^+ \Bd (w + z),\hlambda )} { 
\normta	{\Bd^+
	\Bd  (w+z)}
} 
\frac { 
	\normta	{\Bd^+
		\Bd  (w+z)}}{ 
\normta	{\Bd^+
	\Bd  w}} .
		\end{equation}

		The idea is now to take $z = z^*(w)$, with $z^*(w)$ such that $ \Bd (w + z^*(w))  \in \Range \Embed$. 
		This is obtained by taking $\Bd z^*  = G z^*  = - \Qsigma \Bd w$. 
		We have, with  $\tphi = [\Bd w - \Qsigma \Bd w] = [ \Bd (w + z^*(w))] = [ \Bd w ]$ (we recall that $[\phi]$ stands for the equivalence class of $\phi$), and then
		\begin{equation}\label{eq:A1}
\langle \tphi , \hlambda \rangle = \langle \Bd w - \Qsigma \Bd , \hlambda \rangle  =  \langle \Bd w , \hlambda \rangle =
b(\Bd^+ \Bd w,\hlambda).
\end{equation}
We now point out that for all $\phi \in \Lh'$ we have $\Embed([\phi]) = \Pi_\sigma^T \phi$, which can be seen by comparing \eqref{defK} and \eqref{defPisigma}.
We then can write
\begin{multline}\label{eq:A2}
\widehat m(\tphi,\tphi) = m (\Embed([\Bd (w +  z^*(w))]), 
\Embed([\Bd (w +  z^*(w))])
) \\= m(\Pi_\sigma^T (\Bd (w +  z^*(w)))
,\Pi_\sigma^T (\Bd (w +  z^*(w)))
\\	= \ta(\Bd^+ \Pi_\sigma^T (\Bd (w +  z^*(w))),\Bd^+ \Pi_\sigma^T (\Bd (w +  z^*(w)))) 
\\	=
\ta(\Bd^+ \Bd (w +  z^*(w)),\Bd^+ \Bd (w +  z^*(w))) , 
\end{multline}
where we could write $ \Pi_\sigma^T \Bd (w+z^*(w)) =  \Bd (w+z^*(w))$ thanks to our choice of $z^*(w)$. Plugging \eqref{eq:A1} and \eqref{eq:A2} into \eqref{eq:A3}, and the resulting bound  into \eqref{eq:A4} we obtain
\[
\sqrt{s_h(\hlambda,\hlambda)} \lesssim \sqrt{K^*} \opNorm{\ta}{\ta}{\Pi_d} \sup_{\tphi \in \hLh'} \frac {\langle \tphi, \hlambda \rangle} {\sqrt{	\widehat m(\tphi,\tphi) }}
\sup_{w \in \Vh} \ \frac{\| \Bd^+ \Bd (w + z^*(w)) \|_{\ta} }{\| \Bd^+ \Bd w \|_{\ta}}.
\]
It remains to bound 
\[
\sup_{w \in \Vh} \ \frac{\normta{\Bd^+ \Bd (w + z^*(w)) } }{ \normta{\Bd^+ \Bd w} } = \sup_{w \in \Vh} \frac{\normta{\Bd^+ (\Bd w - \Qsigma \Bd w)} }{ \normta{\Bd^+ \Bd w} }.
\]

We have
\begin{multline*}
 \normta{\Bd^+ (\Bd w - \Qsigma \Bd w) } \leq  \normta{\Bd^+ \Bd w}  +  \normta{\Bd^+ \Qsigma \Bd w  }\\=
 \normta{\Bd^+ \Bd w} + \normta{ \Bd^+\Qsigma \Bd \Bd^+ \Bd w} \\
\lesssim \left(1 + \opNorm{\ta}{\ta} {\Bd^+ \Qsigma \Bd} \right)\normta{ \Bd^+ \Bd w}.
\end{multline*}
This finally yields
\[
\sqrt{s_h(\hlambda,\hlambda)} \lesssim \sqrt{K^*}  \opNorm{\ta}{\ta}{\Pi_d} \left(1 + \opNorm{\ta}{\ta} {\Bd^+ \Qsigma \Bd} \right)
\sup_{\tphi \in \hLh'} \frac {\langle \tphi, \hlambda \rangle} {\sqrt{	\widehat m(\tphi,\tphi) }},
\]
which concludes the proof.

\section{Proof of Lemma \ref{lem:minusprec}}
We start by noting 

\[
\| \phi \|_{s,\G^k}^2 = | \phi |^2_{s,\G^k} + H^{1-2s}  | \bar \phi |^2
\]
and the $H^{-s}$ norm as
\[
\| \phi \|^2_{-s,\G^k} = | \phi |^2_{-s,\G^k} + H^{2s-1} H^{2(d-1)} | \bar \phi |^2
\]
We have
\begin{multline}
	| \phi |_{-1/2,\G^k} = \sup_{\psi \in H^{1/2}(\G^k)} \frac{
		\sum_{F \subset \G^k} \int_{F^k} (\phi \pm \bar\phi^F) \psi }
	{| \psi |_{1/2,\G^k}}\\
	\leq 
	\sup_{\psi \in H^{1/2}(\G^k)} \frac{
		\sum_{F \subset \G^k} \int_{F^k} (\phi - \bar\phi^F) \psi }
	{| \psi |_{1/2,\G^k}} +  \sup_{\psi \in H^{1/2}(\G^k)} \frac{
		\sum_{F \subset \G^k} \int_{F^k} (\bar\phi^F - \bar\phi)\bar \psi^F }
	{| \psi |_{1/2,\G^k}} \\
	\leq \sup_{\psi \in H^{1/2}(\G^k)} \frac{
		\sum_{F \subset \G^k}  | \phi |_{-1/2,F} | \psi |_{1/2,F} }
	{| \psi |_{1/2,\G^k}} +  \sup_{\psi \in H^{1/2}(\G^k)} \frac{
		\sum_{F \subset \G^k} H^{d-1} |\bar\phi^F - \bar\phi |  |\bar \psi^F - \bar \psi |}
	{| \psi |_{1/2,\G^k}}\\
	\lesssim 
	\sqrt{
		\sum_{F \subset \G^k}  | \phi |_{-1/2,F}^2
	} + \sqrt{\sum_{F \subset \G^k} H^{2(d-1)} |\bar\phi^F - \bar\phi |^2}.
\end{multline} 
This leads to
\[
| \phi |_{-1/2,\G^k}^2 \lesssim \sum_{F \subset \G^k}  | \phi |_{-1/2,F}^2 + \sum_{F \subset \G^k} H^{2(d-1)} |\bar\phi^F - \bar\phi |^2.
\]

Let us check the converse bound. There holds
\begin{multline*}
	| \phi |_{-1/2,F} = \sup_{\psi \in H^{1/2}, \fint \phi=0} \frac{\int_{F}  \phi \psi}{| \psi |_{1/2,F}} \lesssim 
	\sup_{\psi \in H^{1/2-\varepsilon}, \fint \phi=0} \frac{\int_{F}  \phi \psi}{| \psi |_{1/2-\varepsilon,F}}\\ \lesssim\underbrace{\frac 1 \varepsilon \sup_{\psi \in H^{1/2}} \frac{\int_{F}  \phi \psi}{ | \psi |_{H^{1/2-\varepsilon}_0(F)}} \lesssim \frac 1 \varepsilon \sup_{\psi \in H^{1/2}_{00}} \frac{\int_{F}  \phi \psi}{ | \psi |_{H^{1/2-\varepsilon}_0(F)}} }_{\text{density of $H^{1/2}_{00}$ in $H^{1/2-\varepsilon}_0$}}
	\lesssim \frac{h^{-\varepsilon}}{\varepsilon} \sup_{\psi \in H^{1/2}_{00}} \frac{\int_{F}  \phi \psi}{ | \psi |_{H^{1/2}_{00}(F)}}
\end{multline*}
whence
\[
| \phi |_{-1/2,F} \lesssim \log(H/h) \| \phi \|_{H^{1/2}_{00}(F)'}.
\]

\end{document}